\documentclass[11pt]{amsart}
\usepackage[margin=3cm]{geometry}

\usepackage{latexsym}
\usepackage{amssymb}
\usepackage{amsmath}
\usepackage{color}
\usepackage{bbm}
\usepackage[normalem]{ulem}
\usepackage{hyperref}

\usepackage{graphicx}
\usepackage{tcolorbox}
\usepackage{tabularx}
\usepackage[all]{xy}

\usepackage{tikz-cd}
\usepackage{xparse}

\tikzset{
  symbol/.style={
    draw=none,
    every to/.append style={
      edge node={node [sloped, allow upside down, auto=false]{$#1$}}}
  }
}

\NewDocumentCommand\DownArrow{O{2.0ex} O{black}}{%
   \mathrel{\tikz[baseline] \draw [<-, line width=0.5pt, #2] (0,0) -- ++(0,#1);}
}
\NewDocumentCommand\UpArrow{O{2.0ex} O{black}}{%
   \mathrel{\tikz[baseline] \draw [->, line width=0.5pt, #2] (0,0) -- ++(0,#1);}
}

\newtheorem{theorem}{Theorem}[section]
\newtheorem{lemma}[theorem]{Lemma}
\newtheorem{proposition}[theorem]{Proposition}
\newtheorem{corollary}[theorem]{Corollary}
\newtheorem{definition}[theorem]{Definition}

\newtheorem{example}[theorem]{Example}
\newtheorem{remark}[theorem]{Remark}

\newcommand{\cl}[1]{\mathcal{#1}}
\newcommand{\bb}[1]{\mathbb{#1}}

\begin{document}

\title[Affiliated operators and control theory]{Affiliated operators for classical and quantum control}

\author[D. Giannakis]{Dimitrios Giannakis}
\address{Department of Mathematics\\ Dartmouth College\\ Hanover \\ NH 03755 \\ USA}
\email{dimitrios.giannakis@dartmouth.edu}

\author[G. Hoefer]{Gage Hoefer}
\address{Department of Mathematics \\ Dartmouth College \\ Hanover \\ NH 03755 \\ USA}\email{gage.hoefer@dartmouth.edu}

\thanks{2020 {\it  Mathematics Subject Classification.} 37C30, 46L51, 46L60, 81Q93, 93B28}

\thanks{{\it Key words and phrases:} Bilinear control systems, Koopman operator, affiliated operators, finite von Neumann algebras, reproducing kernel Hilbert spaces}

\date{\today}

\begin{abstract}
Using techniques from the theory of von Neumann algebras, we propose a framework for addressing questions of controllability of bilinear systems on infinite dimensional Hilbert spaces. In the setup, we assume only that the drift and control terms arising in a bilinear control system are affiliated with a von Neumann algebra of finite type acting on the same Hilbert space. When the control terms satisfy basic norm bound conditions, we prove existence of time-optimal controls. In the more general setting where all operators may be unbounded, we show how the dynamical Lie algebra for the system is still well-defined and may be used to check approximate controllability of the system in question. We discuss how this approach can be applied to classical dynamical systems through the Koopman operator formalism, and investigate potential candidates for the von Neumann algebra which may guide the choice of controls. We illustrate how an affiliation relation naturally arises in both classical and quantum control systems with a few examples. 
\end{abstract}

\maketitle

\section{Introduction}\label{sec:intro}

\subsection{von Neumann algebras and quantum control systems}\label{ss:intro_qcs}
Control theory, which is the study of finding sufficient conditions and devising methods to influence the dynamical behavior of macro and micro-systems in an optimal way, has a long and rich history with applications in diverse areas as chemistry, plasma physics, solid-state technology, information theory, and engineering. A typical closed quantum control system usually extends the Schrödinger equation (oftentimes governed solely by the so-called ``drift" term) by time-dependent control Hamiltonians, all of which act as linear operators on a common Hilbert space $H$. When the controlled part is linear in both the state and the control, the equation of motion describing the modified dynamics fits into the class of bilinear control systems. Explicitly, if $\xi = \xi(t) \in H$ is the time-dependent vector state with initial state $\xi_{0} \in H$, then the equation of motion for the system may be written as
\begin{gather*}
    i\hbar\frac{d}{dt}\xi(t) = (V_{0}+V(t))\xi(t), \;\;\;\; \xi(0) = \xi_{0}.
\end{gather*}
Oftentimes, we may further decompose the time-dependent control $V(t)$ such that the total Hamiltonian guiding the dynamics is of the form $V_{0}+\sum_{j=1}^{N}u_{j}(t)V_{j}$, where $u_{j}: \mathbb{R}\rightarrow \mathbb{R}$ are real-valued functions and each $V_{j}$ acts as a self-adjoint linear operator on $H$, for $j=0, \hdots, N$. Such systems are natural to consider, and have been extensively studied (see, for instance \cite{bbr, bcc, kzs, rbmb, wtl, yl} and the references therein for a non-exhaustive list). 

The broad goal for a quantum control system is to use the controls to manipulate the system over time, steering a given initial (vector) state of the system $\xi_{0} \in H$ to a target state $\xi_{f} \in H$. A controllable system is one in which any given initial state can be guided to an arbitrary state over some finite time interval. If the system is not fully controllable, a natural question is to determine which states in $H$ are accessible under the dynamics from an initial state. The reachable set--- the set of states which the initial state can be guided to--- becomes the primary object of interest when determining controllability of the system. 

When the Hilbert space $H$ is finite-dimensional (say, of dimension $n$) there are well-developed methods for determining the overall controllability of the system and finding time-optimal controls. A well-established criterion for when the system is fully controllable relies on the dynamical Lie algebra of the system; this is the Lie algebra generated using the Hamiltonians $iV_{0}, \hdots, iV_{N}$ considered as $n\times n$ skew-adjoint matrices inside $\mathfrak{u}(n)$. This Lie algebra may also be used to help understand which unitaries in $U(n)$ are available for implementing quantum operations on $H$. 

This approach becomes much more difficult to follow (if not impossible, without placing further assumptions on the system) in the case when $H$ is infinite-dimensional. Many quantum mechanical systems are inherently infinite-dimensional, and the Hamiltonians which arise in these situations often act as densely-defined but unbounded linear operators on the Hilbert space. From the perspective of Lie theory, performing algebraic operations with unbounded operators is problematic: there is no guarantee that if $a, b$ are densely-defined and closed self-adjoint operators on $H$ that $a+b$ or $ab-ba$ are well-defined on $H$. Additionally, there are different topologies one can place on $U(H)$ (the infinite-dimensional analogue of the $n\times n$ unitary group $U(n)$) which make it a topological group: either the topology inherited from the norm in $\mathcal{B}(H)$, or the strong operator topology. An organizing principle in this work is to frame the governability of a generator for a one-parameter unitary semi-group (representing the dynamics of a system on some space) as a control theory problem. While $U(H)$ endowed with the norm topology is a well-defined Banach-Lie group, requiring norm-continuity for a one-parameter unitary semi-group inside $U(H)$ forces the generator of the semi-group to be a bounded operator on $H$. Requiring only strong operator topology continuity for a representation of a unitary group allows for the use of unbounded generators, which is why a more suitable candidate for our purposes would be to try and recover Lie group structure on $U(H)$ under the strong operator topology. 

As a first step in addressing these complications, and with the aim of recovering the dynamical Lie algebra of an infinite-dimensional quantum mechanical system in mind, we place an additional restriction on the Hamiltonians arising in our bilinear control system. As first shown in \cite{am_one}, if $a, b$ are closed and densely-defined linear operators on $H$ which are \textit{affiliated} with a von Neumann algebra $\mathcal{M}$ of \textit{finite type} acting on $H$, then $e^{ita} \in \mathcal{M}$ as a unitary for all $t \in \mathbb{R}$ and the unitary group of $\mathcal{M}$ endowed with the strong operator topology becomes a well-defined, infinite-dimensional (closed) Lie group. Furthermore, using the results of Murray-von Neumann \cite{mvn}, all domain issues disappear when forming algebraic combinations of $a, b$. Specifically, if $L^{0}(\mathcal{M})$ denotes the set of all closed and densely-defined operators on $H$ affiliated with $\mathcal{M}$, then $L^{0}(\mathcal{M})$ is a complete, metrizable topological $*$-algebra. 

If Hamiltonians $V_{0}, \hdots, V_{N}$ are affiliated with some finite von Neumann algebra $\mathcal{M}$, we show how the dynamical Lie algebra for the system is a well-defined real Lie algebra inside $L^{0}(\mathcal{M})$. As $\mathcal{M}$ is a (usually infinite-dimensional) dual Banach space, we are able to utilize techniques from the theory of infinite-dimensional control systems on Banach spaces (see \cite{bdpdm, hls}), along with the operator algebraic properties of $\mathcal{M}$, to establish fundamental results on the properties of the reachable sets for the system and how we can address questions of time-optimal control in this setting. With the recovery of the dynamical Lie algebra, and an approximate version of the well-known Lie Algebra Rank Condition (LARC), the methods developed in the study of abstract geometric control theory (see \cite{jur}) can be applied when analyzing the system. For other work approaching control problems with a von Neumann algebraic approach--- particularly in the context of free probability--- see \cite{gjnp_one, gjnp_two}.

\subsection{von Neumann algebras and Koopman theory}\label{ss:intro_koopman}
The von Neumann algebraic viewpoint also has natural connections to the control of classical dynamical systems, when using an operator-theoretic approach first developed by Koopman and von Neumann (see \cite{kvn}). A common problem which arises for a given measure-preserving, invertible transformation $\Phi: X\rightarrow X$ of a probability space $(X, \Sigma, \mu)$ is that the transformation $\Phi$ which governs the dynamics of the state space $X$ is often highly nonlinear. In order to utilize operator theoretic tools when analyzing behavior of classical dynamical systems, we instead lift to the space of observables $L^{2}(X, \mu)$ and study the induced linear action of $\Phi$ on $L^{2}(X, \mu)$ via the Koopman and transfer operators $Uf := f\circ \Phi$ and $Pf = f\circ \Phi^{-1}$, for $f \in L^{2}(X, \mu)$. These are both unitary operators on the Hilbert space of observables (with $U^{*} = P$). Additionally, by Stone's Theorem (for one-parameter, strongly continuous unitary groups \cite[Chapter 5]{kr_one}) the Koopman group $\{U^{t}\}_{t \in \mathbb{R}}$ of unitaries on $L^{2}(X, \mu)$ induced by a measure-preserving flow $\Phi^{t}: X\rightarrow X, t \in \mathbb{R}$ is completely determined by its generator: this is a densely-defined, closed and skew-adjoint operator $V: {\rm dom}(V)\rightarrow L^{2}(X, \mu)$ given via the strong limit 
\begin{gather*}
    Vf = \lim\limits_{t\rightarrow 0}\frac{U^{t}f-f}{t},
\end{gather*}   
\noindent for which $U^{t} = e^{tV}$ for every $t \in \mathbb{R}$ via the Borel functional calculus. In the continuous-time setting, $V$ (modulo an imaginary factor) may be considered as an analogue of a quantum mechanical Hamiltonian. Therefore, we may reframe the controllability of a classical dynamical system by investigating the corresponding bilinear control system on $L^{2}(X, \mu)$ (also known as control-affine systems), where generator $V$ acts as the time-independent drift term and the time-dependent controls are chosen as skew or self-adjoint operators on $L^{2}(X, \mu)$. 

Recent years have seen a surge of interest in numerical methods for approximating Koopman and transfer operators of dynamical systems, including applications of Koopman operator methods in control; see \cite{BruntonEtAl22, KlusEtAl20, MauroyEtAl20, or} and references therein.
A common theme among these methods is to cast control problems in nonlinear dynamical systems as bilinear control-affine problems formulated in terms of Koopman operators acting on observables \cite{Surana16,WilliamsEtAl16}, in line with what was mentioned just above.   
An early example is the dynamic mode decomposition (DMD) with control technique \cite{ProctorEtAl16}---this method extends the popular DMD method for Koopman operator approximation to build a bilinear model for the dynamics of observables with separate components representing the Koopman operator of the uncontrolled system and the linear control inputs.
This general approach was further developed in the setting of model predictive control \cite{KordaMezic18}, using DMD-based approximations of the Koopman operator to predict the evolution of the components of the state vector.
Other approaches employ spectral decompositions of the Koopman operator \cite{KaiserEtAl21,KordaMezic20} to identify spaces of controlled observables with a predictable time evolution, or expectation maximization techniques \cite{OttoEtAl24, PeitzEtAl20} to estimate control-input dependent families of Koopman operators.
A probabilistic error analysis for variants of these methods in the setting of stochastic dynamics was carried out in \cite{NuskeEtAl23}. 
Koopman operator methods for control have also been developed for partial differential equation systems \cite{ArbabiEtAl18,PeitzKlus19}.

Returning to the example of the measure-preserving flow $\Phi^t$, there is a natural candidate for the finite von Neumann algebra $\mathcal{M}$ on $L^{2}(X, \mu)$ we would expect our drift and controls to be affiliated to: by setting ${\rm VN}(\Phi)$ as the closure under the weak operator topology of the unital $*$-subalgebra in $\mathcal{B}(L^2(X, \mu))$ generated by the Koopman group $\{U^{t}\}_{t \in \mathbb{R}}$, we show that ${\rm VN}(\Phi)$ is not only finite type but abelian, with $V$ affiliated to ${\rm VN}(\Phi)$. We focus particularly on this von Neumann algebra, as it is more tractable to work with when one has access to the action of the generators $U^{t}$ on $H$. Furthermore, when paired with a data-driven approach to approximating the Koopman operators $U^{t}$ (as suggested in \cite{kpm}; see also \cite{gljmpss, gv_one} or the recent survey \cite{gm} and the references therein), this suggests a numerically feasible method for choosing appropriate controls for the system. 

\subsection{Contents} The structure of the manuscript is as follows. In Section \ref{sec:preliminaries}, we recall the necessary background on unbounded operators, von Neumann algebras, and Lie theory. In Section \ref{sec:adm_controls}, we propose a notion of admissible controls for infinite-dimensional quantum systems, allowing for the Hamiltonians of the system to be unbounded but affiliated to a finite tracial von Neumann algebra $\mathcal{M}$. We establish--- assuming certain conditions on the norm of the admissible controls--- properties of the reachable set for the system, and basic existence theorems for time optimal controls, analogous to known results in finite-dimensions. In Section \ref{sec:qcs_geom_control}, we move to the case when the control is dependent only on piecewise constant functions, and so the Hamiltonian decomposes as a finite linear combination of self-adjoint operators. Assuming no regularity or operator norm conditions on the individual Hamiltonians, we show how affiliation allows us to recover the ``dynamical Lie algebra" for the system. We discuss various notions of controllability for such a system, and show how equivalence between these notions no longer coincide as compared to the finite-dimensional setting. We also introduce an operator algebraic informed approximation scheme for unbounded Hamiltonians with complicated spectral behavior. In Section \ref{sec:classical_systems}, we discuss how the affiliated operator approach to control theory may be useful for classical systems; in particular, for a variety of dynamical systems when considering the corresponding generator $V$ of a one-parameter strongly continuous unitary group. We clarify the relation between the Koopman group, and the canonical von Neumann algebra $V$ is affiliated to. We also give a more concrete picture of our approximation scheme in this setting, and connect our results to the recently developed theory of reproducing kernel Hilbert algebras. Finally, in Section \ref{sec:examples} we discuss a few examples which naturally arise. 

While a natural topic to investigate in the control theory literature, we do not address any questions involving stability or feedback. These, and related questions will be reserved for future work.

\section{Preliminaries}\label{sec:preliminaries}
In this section, we collect some definitions and results (many which are well-known) which will be needed throughout. For an excellent introduction to the theory of operator algebras, the reader is directed to \cite{kr_one, takesaki, takesaki_two}. We assume basic familiarity with the theory of Lie groups and Lie algebras; see \cite{bourbaki} for greater detail. Additionally, for an exposition of RKHS theory, see \cite{paulrag}.
\subsection{Measure theory and unbounded operators} Let $H$ be a Hilbert space; all Hilbertian inner products will be linear in the first coordinate. The domain of a linear operator $a$ on $H$ is denoted ${\rm dom}(a)$, and the range is denoted ${\rm rng}(a)$. We say a linear operator $a: {\rm dom}(a) \rightarrow H$ is \textit{closable} if ${\rm dom}(a)$ is a dense subspace of $H$ and if the closure $\overline{\rm Gr(a)}$ of its graph
\begin{gather*}
    {\rm Gr}(a) = \{(\xi, a\xi): \; \xi \in {\rm dom}(a)\} \subseteq H\oplus H,
\end{gather*}
\noindent is the graph of a linear operator, and \textit{closed} if ${\rm Gr}(a)$ is already closed. If $a$ is closable, we write $\overline{a}$ for the closure of $a$. A \textit{core} (sometimes called an \textit{essential domain}) of a closable operator $a$ on $H$ is a subset $D \subseteq {\rm dom}(a)$ such that the closure of the restriction $a|_{D}$ is $\overline{a}$. We say a closed, densely-defined operator $a: {\rm dom}(a)\rightarrow H$ is \textit{positive} if $\langle a\xi, \xi\rangle \geq 0$ for all $\xi \in {\rm dom}(a)$. For (potentially unbounded) densely defined, closed operators $a, b$ on $H$ and $\alpha \in \bb{C}$ we set
\begin{gather*}
    {\rm dom}(\alpha a) := {\rm dom(a)},
    \\ (\alpha a)\xi := \alpha(a\xi), \;\;\;\; \xi \in {\rm dom}(\alpha a),
\end{gather*}
\begin{gather*}
    {\rm dom}(a+b) := {\rm dom}(a) \cap {\rm dom}(b),
    \\ (a+b)\xi := a\xi+b\xi, \;\;\;\; \xi \in {\rm dom}(a+b),
\end{gather*}
\begin{gather*}
    {\rm dom}(ab) := \{\xi \in {\rm dom}(b), b\xi \in {\rm dom}(a)\},
    \\ (ab)\xi := a(b\xi), \;\;\;\; \xi \in {\rm dom}(ab).
\end{gather*}
\noindent A densely defined, closed operator $a$ on $H$ is called \textit{resolvent class} if: 
\begin{itemize}
    \item[(i)] there exist self-adjoint operators $a', b'$ on $H$ such that ${\rm dom}(a') \cap {\rm dom}(b')$ is a core of $a'$ and $b'$;
    \item[(ii)] $a = \overline{a'+ib'}, a^{*} = \overline{a'-ib'}$.
\end{itemize}
\noindent In this case, $a', b'$ are uniquely determined and we may write
\begin{gather*}
    {\rm Re}(a) := a' = \frac{1}{2}\overline{a+a^{*}}, \;\;\;\; {\rm Im}(a) := b' = \frac{1}{2i}\overline{a-a^{*}}.
\end{gather*}
\noindent We endow the space of all resolvent class operators with the \textit{strong resolvent topology (SRT)}: we say a net $\{a_{\lambda}\}_{\lambda \in \Lambda}$ converges in the SRT to $a$ if and only if
\begin{gather*}
    ({\rm Re}(a_{\lambda})-i)^{-1}\xi\rightarrow ({\rm Re}(a)-i)^{-1}\xi, \;\;\;\; ({\rm Im}(a_{\lambda})-i)^{-1}\xi\rightarrow ({\rm Im}(a)-i)^{-1}\xi,
\end{gather*}
\noindent for every $\xi \in H$. 

In what follows, let $X$ be a locally compact, $\sigma$-compact Hausdorff space with Radon measure $\mu$ on $X$. We let $\cl{L}(X), \cl{B}(X), \cl{B}_{b}(X)$ and $\cl{N}(X)$ denote the spaces of measurable, Borel measurable, bounded Borel, and null functions on $X$, respectively. When $(X, \Sigma, \mu)$ is a $\sigma$-finite measure space, for $1 \leq p \leq \infty$ we let $L^{p}(X, \mu)$ denote the associated $L^{p}$-space. Given a Banach space $Y$ and $1 \leq p \leq \infty$, let $L^{p}(X, Y)$ denote the Bochner space of all measurable functions $F: X\rightarrow Y$ (modulo functions which are almost everywhere zero) such that the norm function $t \mapsto \|F(t)\|_{Y}$ belongs to $L^{p}(X, \mu)$. This is a Banach space under the norm $\|F\|_{p}$, with the latter taken as the $L^{p}(X, \mu)$ norm of $\|F(\cdot)\|_{Y}$. When $Y$ is separable, we say a function $f: X\rightarrow Y^{*}$ is \textit{${wk}^{*}$-measurable} if for all $y \in Y$, the function $t\mapsto \langle f(t), y\rangle$ is measurable over $X$, where $\langle \cdot, \cdot\rangle$ denotes the dual pairing of $Y$ and $Y^{*}$; the same notation will be used for Hilbertian inner products (and it will be clear from the context which of these two uses is intended). If $f: X\rightarrow Y^{*}$ is ${\rm wk}^{*}$-measurable, then the norm $t \mapsto \|f(t)\|$ is measurable on $X$. Let $L_{\sigma}^{\infty}(X, Y^{*})$ denote the space of all ${\rm wk}^{*}$-measurable functions $f: X\rightarrow Y^{*}$ (again, up to almost everywhere zero functions) such that $\|f(\cdot)\|$ is essentially bounded. This is also a Banach space, under the norm $\|f\| = \big\|\|f(\cdot)\|\big\|_{L^{\infty}(X, \mu)}$. Furthermore, by the Dunford-Pettis Theorem (see \cite{dp}) we have the isometric isomorphism
\begin{gather}\label{eqn_dp_iso}
    L_{\sigma}^{\infty}(X, Y^{*}) \cong L^{1}(X, Y)^{*},\;\;\;\;
    \langle f, F\rangle := \int\limits_{X}\langle f(t), F(t)\rangle d\mu(t).
\end{gather}

\subsection{Reproducing kernel Hilbert algebras}\label{ss:rkha} The following material can be found in \cite{gm_rkha} (see also \cite{dg, gljmpss, gm}). Let $\mathcal{H}$ be an RKHS of complex-valued functions on a set $X$, with reproducing kernel $k: X\times X\rightarrow \mathbb{C}$ and inner product denoted by $\langle \cdot, \cdot\rangle_{\mathcal{H}}$. We write $k_{x} := k(x, \cdot)$ for the kernel section at element $x \in X$, with $\delta_{x}: \mathcal{H}\rightarrow \mathbb{C}$ the corresponding pointwise evaluation functional $\delta_{x}(f) = f(x) = \langle f, k_{x}\rangle_{\mathcal{H}}$. We also let $\varphi: X\rightarrow \mathcal{H}$ denote the canonical feature map, where $\varphi(x) = k_{x}$ for $x \in X$. As first (formally) introduced in \cite{gm_rkha}, the RKHS $\mathcal{H}$ is a \textit{reproducing kernel Hilbert algebra (RKHA)} if the mapping $k_{x} \mapsto k_{x}\otimes k_{x}$, for $x \in X$, extends to a bounded linear map $\Delta: \mathcal{H}\rightarrow \mathcal{H}\otimes \mathcal{H}$. In this case, $\Delta^{*}$ (the formal adjoint of $\Delta$) is a bounded linear map which implements pointwise multiplication: indeed, for $f, g \in \mathcal{H}$ and $x \in X$ we check
\begin{gather*}
    \langle \Delta^{*}(f\otimes g), k_{x}\rangle_{\mathcal{H}} = \langle f\otimes g, \Delta(k_{x})\rangle_{\mathcal{H}} = \langle f\otimes g, k_{x}\otimes k_{x}\rangle_{\mathcal{H}} \\= \langle f, k_{x}\rangle_{\mathcal{H}}\langle g, k_{x}\rangle_{\mathcal{H}} = f(x)g(x).
\end{gather*}
\noindent Thus, $\mathcal{H}$ is both a function space with Hilbertian structure, and a commutative algebra with respect to pointwise function multiplication. 

If $\pi: \mathcal{H}\rightarrow \mathcal{B}(\mathcal{H})$ denotes the representation sending $f \mapsto \pi(f)$ where $\pi(f)g := fg$ for $g \in \mathcal{H}$, one may check that for norm $\|f\|_{\rm op} := \|\pi(f)\|_{\mathcal{B}(\mathcal{H})}$, $\|fg\|_{\rm op} \leq \|f\|_{\rm op}\|g\|_{\rm op}$ for all $f, g \in \mathcal{H}$ and that $\|\cdot\|_{\rm op}$ generates a coarser topology on $\mathcal{H}$ than the topology generated by $\|\cdot\|_{\mathcal{H}}$. If $k$ is real-valued, we may equip $\mathcal{H}$ with a pointwise complex conjugation $*: \mathcal{H}\rightarrow \mathcal{H}$, where $f^{*} := \overline{f}$ for $f \in \mathcal{H}$, turning $\mathcal{H}$ into a Banach $*$-algebra. If the constant function $1_{X}: X\rightarrow \mathbb{C}$ where $1(x) = 1$ for any $x \in X$ is in $\mathcal{H}$, then $\mathcal{H}$ is a unital Banach $*$-algebra; furthermore, the Hilbert space and operator norms become equivalent, as
\begin{gather*}
    \frac{1}{\|1_{X}\|_{\mathcal{H}}}\|f\|_{\mathcal{H}} \leq \|f\|_{\rm op} \leq \|\Delta\|_{\mathcal{B}(\mathcal{H})}\|f\|_{\mathcal{H}}.
\end{gather*}
\noindent We note that an RKHA $\mathcal{H}$ is also a co-commutative, co-associative co-algebra with $\Delta: \mathcal{H}\rightarrow \mathcal{H}\otimes \mathcal{H}$ functioning as the co-multiplication operator. Co-associativity (i.e., the formal equality $(\Delta\otimes {\rm id})\circ \Delta = ({\rm id}\otimes \Delta)\circ \Delta$) follows from associativity of multiplication $\Delta^{*}$. Recall that morphisms between reproducing kernel Hilbert spaces $(X_{1}, \mathcal{H}_{1})$ and $(X_{2}, \mathcal{H}_{2})$ take the form of bounded linear maps $T: \mathcal{H}_{1}\rightarrow \mathcal{H}_{2}$ such that $T(k_{x}) = k_{F(x)}$ for some set mapping $F: X_{1}\rightarrow X_{2}$. As mentioned in \cite[Section 2]{gm}, when $\mathcal{H}_{i}$ is an RKHA over $X_{i}, i = 1, 2$ a morphism $T: \mathcal{H}_{1}\rightarrow \mathcal{H}_{2}$ on the RKHS level is equivalent to the existence of a Banach algebra homomorphism $T^{*}: \mathcal{H}_{2}\rightarrow \mathcal{H}_{1}$ (i.e., that $T^{*}$ is a multiplicative mapping). We will utilize this fact in Section \ref{sec:classical_systems}. 

\subsection{Operator algebras and completely positive maps}\label{subsec:operator_algebras_cp_maps}
We denote by $\cl{B}(H)$ the space of all bounded linear operators on $H$, with $I_{H}$ (or sometimes just $I$ if the context is clear) the identity operator on $H$. Additionally, we let $\mathcal{T}(H)$ denote the space of trace-class (with respect to ${\rm Tr}$ on $\mathcal{B}(H)$) operators. If $a \in \mathcal{B}(H)$ is a bounded linear operator, we let $\sigma(a)$ denote the \textit{spectrum} of $a$ inside $\mathcal{B}(H)$: this is the set 
\begin{gather*}
    \sigma(a) = \{\lambda \in \mathbb{C}: \; (a-\lambda I) \; {\rm is \; not \; invertible \; on \;} H\}.
\end{gather*}
Similarly, we let $\rho(a) = \mathbb{C}\setminus \sigma(a)$ denote the \textit{resolvent} of $a$. For $z \in \rho(a)$, we let $R_{z}(a) = (a-zI)^{-1}$. A \textit{C*-algebra} is any $*$-subalgebra $\mathcal{A} \subseteq \mathcal{B}(H)$ closed under the norm topology. The convex cone of positive elements in $\mathcal{A}$ is denoted $\mathcal{A}^{+}$, and for any $n \in \mathbb{N}$ we write $M_{n}(\mathcal{A})$ to denote the ${\rm C}^{*}$-algebra of all $n\times n$ matrices with entries in $\mathcal{A}$. For self-adjoint elements $a, b \in \mathcal{A}$ we write $a \leq b$ if $b-a \in \mathcal{A}^{+}$. A linear map $\Phi: \cl{A}\rightarrow \cl{B}$ between ${\rm C}^{*}$-algebras is called \textit{positive}, if $\Phi(\cl{A}^{+}) \subseteq \cl{B}^{+}$, and \textit{completely positive} if the amplification map $\Phi^{(n)}: M_{n}(\cl{A})\rightarrow M_{n}(\cl{B})$ given by $\Phi^{(n)}((x_{ij})_{i,j}) = (\Phi(x_{ij}))_{i,j}$ is positive for every $n \in \mathbb{N}$. A \textit{state} on a ${\rm C}^{*}$-algebra is a linear functional $\phi: \mathcal{A}\rightarrow \mathbb{C}$ such that $\phi(\mathcal{A}^{+}) \subseteq \mathbb{R}_{+}$, and $\|\phi\| = 1$. In the case where $\mathcal{A}$ is unital, every state $\phi: \mathcal{A}\rightarrow \mathbb{C}$ satisfies $\phi(1_{\mathcal{A}}) = 1$. 

A \textit{von Neumann algebra} is any unital $*$-subalgebra $\mathcal{M}\subseteq \mathcal{B}(H)$ closed under the strong (equivalently, weak) operator topology on $H$. Unless otherwise specified, we assume the unit of $\mathcal{M}$ to be the same as $I$, the identity on $H$. The commutant of a von Neumann algebra $\mathcal{M}\subseteq \cl{B}(H)$ will be denoted by $\mathcal{M}'$, and its Banach space predual by $\mathcal{M}_{\ast}$. Note that every von Neumann algebra is automatically a ${\rm C}^{*}$-algebra, and by the Bicommutant Theorem we have $\mathcal{M}= \mathcal{M}''$ (see \cite[Chapter II]{takesaki}). We call $\mathcal{M}$ a \textit{factor} if $\mathcal{M}\cap \mathcal{M}' = \mathbb{C}I$. The group of all unitary operators in $\mathcal{M}$ is denoted by $U(\mathcal{M})$, and the lattice of projections in $\mathcal{M}$ is denoted $\cl{P}(\mathcal{M})$. If $K \subseteq H$ is a closed subspace, we let $p_{K}$ denote the orthogonal projection onto $K$. If $p \in \cl{P}(\cl{M})$, we let $p^{\perp} = I-p$. We say $\mathcal{M}$ is \textit{$\sigma$-finite}, or \textit{countably decomposable}, if there are at most countably many non-zero orthogonal projections in $\mathcal{P}(\mathcal{M})$. 

For von Neumann algebras $\mathcal{M}, \mathcal{N}$, a linear map $\Phi: \mathcal{M}\rightarrow \mathcal{N}$ is \textit{normal} if it is continuous with respect to the weak topologies on each respective space. If $\mathcal{N} \subseteq \mathcal{M}$ is a subalgebra of $\mathcal{M}$, a \textit{conditional expectation} from $\mathcal{M}$ to $\mathcal{N}$ is a linear map $E: \mathcal{M}\rightarrow \mathcal{N}$ such that:
\begin{itemize}
    \item[(i)] $E$ is positive, in that $E(\mathcal{M}^{+}) \subseteq \cl{N}^{+}$;
    \item[(ii)] $E(n) = n$ for all $n \in \cl{N}$;
    \item[(iii)] $E(n_{1}mn_{2}) = n_{1}E(m)n_{2}$ for all $n_{1}, n_{2} \in \mathcal{N}$ and $m \in \mathcal{M}$.
\end{itemize}
Thus, $E$ is a positive projection from $\mathcal{M}$ onto $\mathcal{N}$ (so that $E\circ E(m) = E(m)$ for all $m \in \mathcal{M}$), and is a bimodule map over $\mathcal{N}$ (i.e., condition (iii) above).

A von Neumann algebra $\mathcal{M}$ is said to be \textit{injective} if for any unital ${\rm C}^{*}$-algebra $\mathcal{A}$ and for any norm-closed, self-adjoint unital subspace $\mathcal{S} \subseteq \mathcal{A}$, any unital, completely positive map $\phi: \mathcal{S}\rightarrow \mathcal{M}$ extends to a unital completely positive map $\phi: \mathcal{A}\rightarrow M$. Under Arveson's Extension Theorem, this is equivalent to the existence of a conditional expectation $E: \mathcal{B}(H)\rightarrow \mathcal{M}$ when $\mathcal{M}\subseteq \mathcal{B}(H)$ is represented faithfully on some Hilbert space $H$ (for the general idea of the proof, see \cite[Proposition 15.1]{paulsen}). If, in addition, $\mathcal{M}_{\ast}$ is separable, we say it is \textit{approximately finite-dimensional (AFD)} if there exists an increasing sequence of finite-dimensional $*$-sub-algebras of $\mathcal{M}$ whose union is dense in $\mathcal{M}$ under the strong operator topology. The seminal work of Connes (see \cite{connes}) showed that when $\mathcal{M}$ has a separable predual, being injective and AFD are equivalent.

\subsection{Traces and affiliated operators} If $\mathcal{M}$ is a von Neumann algebra on $H$, a \textit{trace} is a map $\tau: \mathcal{M}^{+}\rightarrow [0, \infty]$ satisfying:
\begin{itemize}
    \item[(i)] $\tau(a+b) = \tau(a)+\tau(b)$, for $a, b\in \mathcal{M}^{+}$;
    \item[(ii)] $\tau(\lambda a) = \lambda \tau(a)$, for all $\lambda \in [0, \infty)$ and $a \in \mathcal{M}^{+}$;
    \item[(iii)] $\tau(a^{*}a) = \tau(aa^{*})$, for all $a \in \mathcal{M}$.
\end{itemize}
\noindent The trace $\tau$ is \textit{normal} if $\sup_{\lambda}\tau(a_{\lambda}) = \tau(\sup_{\lambda}a_{\lambda})$ for any bounded, increasing net $(a_{\lambda})_{\lambda \in \Lambda}$ in $\mathcal{M}^{+}$ (which is equivalent to the previous definition for a normal linear map), \textit{semifinite} if for any non-zero $a \in \mathcal{M}^{+}$ there exists a non-zero $b \in \mathcal{M}^{+}$ such that $b \leq a$ and $\tau(b) < \infty$, and \textit{faithful} if $\tau(a^{*}a) = 0$ implies $a = 0$, for $a \in \mathcal{M}$. If $\tau(1) < \infty$ (where $1$ denotes the identity of $\mathcal{M}$), $\tau$ is said to be \textit{finite}. When $\tau$ is finite, we automatically normalize so that $\tau(1) = 1$, making $\tau$ a proper state on $\mathcal{M}$. A subspace $K \subseteq H$ is called \textit{$\tau$-dense} (or \textit{completely dense}) for a finite von Neumann algebra $(\mathcal{M}, \tau)$ if there exists an increasing net $(p_{\alpha})_{\alpha} \subseteq \mathcal{P}(\mathcal{M})$ such that $p_{\alpha} \nearrow I$ (strongly), and $p_{\alpha}H \subseteq K$ for any $\alpha$. 

For a faithful, normal, semifinite trace $\tau$ on $\mathcal{M}$, we write $(\mathcal{M}, \tau)$ and say $\mathcal{M}$ is \textit{semifinite}. It can be shown (see \cite[Chapter 5]{takesaki}) for a von Neumann algebra $\mathcal{M}$ that if $\mathcal{M}$ admits a finite, normal, faithful tracial state $\tau$, then $\mathcal{M}$ is \textit{finite} (in terms of the type classification); that is, there are no non-unitary isometries in $\mathcal{M}$. From now on, we make this identification between finite traces and finite type von Neumann algebras without further mention. 

For the benefit of the reader, we sketch the construction of the \textit{GNS representation of $\cl{M}$ using $\tau$}; for greater detail on these points, see \cite[Chapter I.9]{takesaki}. Define a sesquilinear form $\langle \cdot, \cdot\rangle_{\tau}$ on $\mathcal{M}$ via $\langle a, b\rangle_{\tau} := \tau(b^{*}a)$, for $a, b \in \mathcal{M}$. Set $N_{\tau} := \{a \in \cl{M}: \; \tau(a^{*}a) = 0\}$; with an application of the Cauchy-Schwarz inequality, $N_{\tau}$ may be realized as the subspace of all $a \in \cl{M}$ such that $\langle a, b \rangle_{\tau} = 0$ for all $b \in \mathcal{M}$. Define $L^{2}(\cl{M}, \tau)$ as the completion of the pre-Hilbert space $\cl{M}/N_{\tau}$ with respect to the inner product $\langle \hat{a}, \hat{b}\rangle_{\tau} = \tau(b^{*}a)$ (assuming $\hat{a}, \hat{b}$ are in the equivalence classes of $a, b$ under the quotient, respectively). As $\tau$ is faithful and normal, there exists a faithful, normal $*$-homomorphism $\pi_{\tau}: \cl{M}\rightarrow \mathcal{B}(L^{2}(\cl{M}, \tau))$ given by $\pi_{\tau}(a)\hat{b} = \widehat{ab}$. With $\xi_{\tau} = \hat{I}$ in $L^{2}(\cl{M}, \tau)$ one may compute $\tau(a) = \langle \pi_{\tau}(a)\xi_{\tau}, \xi_{\tau}\rangle_{\tau}$ for all $a \in \cl{M}$. As $\tau$ is faithful, $\xi_{\tau}$ is both cyclic and separating for the von Neumann algebra $\pi_{\tau}(\cl{M})$. We say that the GNS representation of $\mathcal{M}$ using $\tau$ is the \textit{standard representation $(\mathcal{M}, \tau)$}, and that $\mathcal{M}$ acting on $L^{2}(\mathcal{M}, \tau)$ is in \textit{standard form}.

If $a$ is a densely defined, closable operator on $H$, we say $a$ is \textit{affiliated with $\mathcal{M}$}, and write $a \; \eta \; \mathcal{M}$, if for every unitary operator $u \in U(\mathcal{M}')$, we have $ua = au$. The set of all densely defined, closed operators affiliated with $\mathcal{M}$ is denoted by $L^{0}(\mathcal{M})$, and each element in $L^{0}(\mathcal{M})$ is called an \textit{affiliated operator}. If $\mathcal{M}$ is, in addition, endowed with a tracial state $\tau$ we often will write $L^{0}(\mathcal{M}, \tau)$ for the set of all operators affiliated with $\mathcal{M}$.

\begin{remark}
\rm Note that if $a$ is densely defined, closable on $H$, and affiliated with $\mathcal{M}$, then so is $\overline{a}$. Additionally, as $ua = au$ this means $u({\rm dom}(a)) = {\rm dom}(a)$ for every $u \in U(\mathcal{M}')$. 
\end{remark}

Suppose $a: {\rm dom}(a) \rightarrow H$ is a closed, densely-defined and normal operator. The set $\{a\}'$ of all bounded linear operators on $H$ which commute with $a$ is a strongly-closed subspace (where strong closure follows from the closure of $a$). Similarly, $\{a^{*}\}'$ is strongly-closed, with $\{a\}' \cap \{a^{*}\}'$ a strongly-closed, unital $*$-subalgebra--- and hence, a von Neumann algebra. We set
\begin{gather*}
    W^{*}(a) := \bigg(\{a\}' \cap \{a^{*}\}'\bigg)',
\end{gather*}
\noindent and say that $W^{*}(a)$ is the \textit{von Neumann algebra generated by $a$}. Note that, by construction, $x \in L^{0}(W^{*}(a))$. 

Let $a$ be affiliated with $\mathcal{M}$, and take 
\begin{gather*}
    a = v|a|, \;\;\;\; |a| = \int\limits_{0}^{\infty}\lambda d e(\lambda),
\end{gather*}
\noindent where the former is the polar decomposition of $a$ and the latter is the spectral decomposition of $|a|$. Note here that $e(\lambda) = e([0, \lambda])$, the spectral projection onto $[0, \lambda]$. We say that $a$ is \textit{$\tau$-measurable} if $\lim\limits_{\lambda\rightarrow \infty}\tau(e(\lambda)^{\perp}) = 0$, or (equivalently) that $\tau(e(\lambda)^{\perp}) < \infty$ for large $\lambda > 0$. We can then define \textit{generalized singular numbers} of $a$, with
\begin{gather*}
    \mu_{t}(a) = \inf\{\lambda > 0: \; \tau(e_{\lambda}(|a|)) \leq t\}, \;\;\;\; t > 0.
\end{gather*} 
We can define the continuous, real valued function $\mu(a): (0, \infty)\rightarrow (0, \infty)$ via $t\mapsto \mu_{t}(|a|)$, for any $\tau$ measurable operator $a \in L^{0}(\mathcal{M}, \tau)$. 

Let $L(\mathcal{M}, \tau)$ denote the collection of all $\tau$-measurable operators of $\mathcal{M}$; then $L(\mathcal{M}, \tau)$ becomes a (metrizable) topological $*$-algebra (see \cite{nelson_three}) with the topology generated by a basis of neighborhoods at $0$
\begin{gather*}
    V(\epsilon, \delta) = \{a \in L(\mathcal{M}, \tau): \; \mu_{\epsilon}(x) \leq \delta\}, \;\;\;\; \epsilon, \delta > 0.
\end{gather*}
\noindent Convergence in this topology is called \textit{convergence in measure}; in the case when $\tau(1) < \infty$, then $L^{0}(\mathcal{M}, \tau) = L(\mathcal{M}, \tau)$ automatically, as for any $x \in L^{0}(\mathcal{M}, \tau)$ we have $\tau(e(\lambda)^{\perp}) < \infty$ for some $\lambda$. Furthermore, if $\mathcal{M}$ acts on a separable Hilbert space it is shown (see \cite{am_one}) that convergence in measure $a_{n} \rightarrow a$ as $n\rightarrow \infty$ for $(a_{n})_{n=1}^{\infty} \subseteq L(\mathcal{M}, \tau)$ and $a \in L(\mathcal{M}, \tau)$ is equivalent to convergence in the strong-resolvent topology of $a_{n}\rightarrow_{\rm SRT} a$. As we will restrict ourselves to dealing only with finite tracial von Neumann algebras in the sequel, we assume all affiliated operators will be $\tau$-measurable with respect to some finite trace. 

The trace $\tau$ extends to a positive tracial functional on $L(\mathcal{M}, \tau)^{+}$ satisfying
\begin{gather*}
    \tau(a) = \int\limits_{0}^{\infty}\mu_{t}(a)dt, \;\;\;\; a\in L(\mathcal{M}, \tau)^{+}.
\end{gather*}
\noindent We then define
\begin{gather*}
    L^{1}(\mathcal{M}, \tau) = \{a \in L(\mathcal{M}, \tau): \; \tau(|a|) < \infty\}, \;\;\;\; \|a\|_{1} = \tau(|a|).
\end{gather*}
\noindent It is known (see \cite[Chapter IX]{takesaki_two}, for instance) that $\|\cdot\|_{1}$ is a norm with $L^{1}(\mathcal{M}, \tau)$ a Banach space. Directly, potentially unbounded (as an operator on a Hilbert space) $a \in L^{1}(\mathcal{M}, \tau)$ if and only if $\mu(a) \in L^{1}(0, \infty)$ and $\|a\|_{1} = \|\mu(a)\|_{L^{1}(0, \infty)}$. As $\mu(a) = \mu(a^{*}) = \mu(|a|)$, we have that $a \in L^{1}(\mathcal{M}, \tau)$ if and only if $a^{*} \in L^{1}(\mathcal{M}, \tau)$, and $\|a\|_{1} = \|a^{*}\|_{1}$. Furthermore, we have an isometric identification of $L^{1}(\mathcal{M}, \tau)$ with $\mathcal{M}_{\ast}$, the pre-dual of $\mathcal{M}$, via the following: for $a \in \mathcal{M}, b \in L^{1}(\mathcal{M}, \tau)$ as $\tau$ extends to a continuous functional on $L^{1}(\mathcal{M}, \tau)$ then
\begin{gather*}
    \mathcal{M}\times L^{1}(\mathcal{M}, \tau) \rightarrow \mathbb{C}, \;\;\;\; (a, b) \mapsto \tau(ab),
\end{gather*}
is well-defined. Equivalently, for every $b \in L^{1}(\mathcal{M}, \tau)$ there exists an $\omega_{b} \in \mathcal{M}_{\ast}$ such that $\omega_{b}(a) = \tau(ab)$ for $a \in \mathcal{M}$.  

While the following two results are well-known, we provide a proof for the benefit of the reader. 
\begin{proposition}
Let $\mathcal{M}$ be a von Neumann algebra acting on Hilbert space $H$. Then $U(\mathcal{M})$ is closed under the strong operator topology if and only if $\mathcal{M}$ is finite. 
\end{proposition}
\begin{proof}
First, assume $\mathcal{M}$ is a finite von Neumann algebra. Suppose that $v \in \mathcal{M}$ is a strong limit of unitaries in $\mathcal{M}$; that is, 
\begin{gather*}
    v = {\rm s}-\lim\limits_{n}u_{n}, \;\;\;\; u_{n} \in U(\mathcal{M}).
\end{gather*}
\noindent We claim that $v$ is an isometry; indeed, for any $\xi \in H$, we have
\begin{gather*}
    \|v(\xi)\| = \lim\limits_{n}\|u_{n}(\xi)\| = \|\xi\|,
\end{gather*}
\noindent as each $u_{n} \in U(\mathcal{M})$. Then as $v \in \mathcal{M}$, with $\mathcal{M}$ finite, $v$ must be unitary. Hence, $\overline{U(\mathcal{M})}^{\rm SOT} = U(\mathcal{M})$. 

Now, assume that $U(\mathcal{M})$ is closed in the strong operator topology; we show that $\mathcal{M}$ must be finite. By the first half of the argument in the previous paragraph, any SOT-limit of unitaries in $\mathcal{M}$ is an isometry; furthermore, an application of the Kaplansky Density Theorem (see \cite[Chapter I.10]{takesaki}) shows that every isometry in $\mathcal{M}$ arises as the SOT-limit of unitaries in $\mathcal{M}$. As we assume $U(\mathcal{M})$ is closed in the strong operator topology, this implies $v$ must be a unitary as well. Now, if $\mathcal{M}$ were infinite type, there exists a projection $p \in \mathcal{P}(\mathcal{M})$ and a partial isometry $v \in \mathcal{M}$ such that $p < I$ (as projections) with $p = vv^{*}$ and $v^{*}v = I$. By the latter equality, $v$ is an isometry, and thus must be unitary by our assumption on $U(\mathcal{M})$. However, if $v$ were unitary we would have $I = v^{*}v = p < I$, a contradiction. Thus, $\mathcal{M}$ must be finite type, as claimed.
\end{proof}
We clarify that in the following result, by ``separability of $\mathcal{M}$" we mean as a Banach space under the strong operator topology, and not in the norm topology. 
\begin{lemma}
Let $(\mathcal{M}, \tau)$ be finite, and assume $M_{\ast}$ is separable. Then $L^{1}(\mathcal{M}, \tau)$ and $\mathcal{M}$ are separable.
\end{lemma}\label{lem:sep_of_lp}
\begin{proof}
Under the isometric identification $L^{1}(\mathcal{M}, \tau) \cong M_{\ast}$, the case when $p = 1$ is trivial. Separability of $\mathcal{M}_{\ast}$ implies that $\mathcal{M}$ can be faithfully represented on a separable Hilbert space $H$; without loss of generality, we may assume we represent $\mathcal{M}$ on $L^{2}(\mathcal{M}, \tau)$ which is separable. The unit ball of $\mathcal{B}(L^{2}(\mathcal{M}, \tau))$ is metrizable in the strong operator topology in this case, which implies $\mathcal{M}_{1}$ (the unit ball of $\mathcal{M}$) is metrizable and separable in the strong operator topology. This implies separability of $\mathcal{M}$ under the strong operator topology, as the countable union of (scaled) unit balls. 
\end{proof}
The following is also an easily verified fact. 
\begin{lemma}\label{l_spatial_isomorphisms_unitaries}
Let $\mathcal{M}\subseteq \cl{B}(H), \mathcal{N}\subseteq \cl{B}(K)$ be (faithfully represented) von Neumann algebras. If $v: H\rightarrow K$ is a unitary, and $\mathcal{M}\cong \mathcal{N}$ via the spatial isomorphism $\mathcal{M}= v^{*}\mathcal{N}v$ on $H$, then $U(\mathcal{M}) = v^{*}U(\mathcal{N})v$.
\end{lemma}
The following are a summary of results from \cite{mvn}, the most important of which we use extensively.
\begin{theorem}[Murray-von Neumann]\label{th:aff_op_algebra}
For an arbitrary finite von Neumann algebra $(\mathcal{M}, \tau)$, the set $L^{0}(\mathcal{M}, \tau)$ forms a $*$-algebra of unbounded operators, where the algebraic operations are defined via:
\begin{gather*}
    (a, b) \mapsto \overline{a+b}, \;\;\;\; (\alpha, a) \mapsto \overline{\alpha a},
    \\ (a, b) \mapsto \overline{ab}, \;\;\;\; a \mapsto a^{*}.
\end{gather*}
\end{theorem}
\begin{remark}\label{rem:completely_dense}
\rm If $a \; \eta \; \mathcal{M}$ when $\mathcal{M}$ is finite, ${\rm dom}(a)$ is a $\tau$-dense subspace in $H$; furthermore, if $(K_{i})_{i=1}^{\infty}$ is a sequence of $\tau$-dense subspaces in $H$ for $\mathcal{M}$, then $\cap_{i=1}^{\infty}K_{i}$ remains completely dense for $\mathcal{M}$. The proof of these statements in the case when $\mathcal{M}$ is a finite factor is given in \cite[Section 16]{mvn}; for the non-factor case, see \cite[Proposition 2.10]{am_one}. As a result, if we pick $a_{1}, \hdots, a_{N} \in L^{0}(\mathcal{M})$ by setting $D := \cap_{i=1}^{N}{\rm dom}(a_{i}) \subseteq H$ we have that $D$ is a $\tau$-dense subspace for $\mathcal{M}$ (and thus, is a dense subspace in $H$) on which $a_{1}, \hdots, a_{N}$ are all densely-defined. Furthermore, directly by definition any linear combination of $a_{1}, \hdots, a_{N}$ are densely defined on $D$. 
\end{remark}
Via \cite[Lemma 3.12]{am_one} we have the following result (which we paraphrase slightly).
\begin{proposition}
Let $(\mathcal{M}, \tau)$ be a countably decomposable finite von Neumann algebra acting on a Hilbert space $H$. Then $L^{0}(\mathcal{M}, \tau)$, endowed with the strong-resolvent topology (SRT), is a complete topological $*$-algebra.
\end{proposition}

\begin{remark}\label{rem:aff_op_bimodule_property}
\rm If $\mathcal{M}$ is finite and acts on $H$, $L^{0}(\mathcal{M})$ is a bimodule over $\mathcal{M}$. This follows immediately from Theorem \ref{th:aff_op_algebra} and the observation that $\mathcal{M}\subseteq L^{0}(\mathcal{M})$. The latter fact holds as every $a \in \mathcal{M}$ is bounded and hence closed and densely-defined, and commutes with all $u \in U(\mathcal{M}')$ by definition of the commutant. 
\end{remark}

As explained in \cite[Section 5.3]{pedersen}, there exists a Borel measure $\mu$ on $\sigma(a)$ such that we may extend the Borel functional calculus to unbounded operators on $H$; that is, the mapping $f \mapsto f(a)$ is an essential $*$-homomorphism from $\cl{B}(\sigma(a))/\cl{N}(\sigma(a))$ into a $*$-algebra of normal operators affiliated to $W^{*}(a)$.  

\begin{lemma}\label{lemma:closure_under_borel}
Let $a$ be a closed, densely-defined operator on $H$ affiliated to a von Neumann algebra $\mathcal{M}\subseteq \cl{B}(H)$. If $a$ is skew or self-adjoint, then $f(a) \; \eta \; \mathcal{M}$ for any $f \in \cl{B}(\sigma(a))$.  
\end{lemma}
\begin{proof}
Whenever $f \in \cl{B}(\sigma(a))$ then $f(a)$ is affiliated with $W^{*}(a)$ via previous discussion (see also \cite[Theorem 1.4]{bl}).  As $a$ is skew or self-adjoint, then $W^{*}(a) = \{a\}''$. Furthermore, it has the property of being the smallest von Neumann algebra such that $a$ is affiliated with it; thus, $W^{*}(a) \subseteq \mathcal{M}$. In $\cl{B}(H)$, we then have $\mathcal{M}'\subseteq W^{*}(a)'$; as $uf(a)u^{*} = f(a)$ for any unitary $u \in W^{*}(a)'$, and as $\mathcal{M}' \subseteq W^{*}(a)'$, then $uf(a)u^{*} = f(a)$ for any unitary $u \in \mathcal{M}'$ as well. Thus, $f(a) \; \eta \; \mathcal{M}$, as claimed.  
\end{proof}

Having developed the necessary prerequisites, we connect Lie theory to our discussion of von Neumann algebras. The following lemma (in different forms) appears in \cite{kato, nelson_two}.
\begin{lemma}\label{l_tk_nelson}
Let $a, b$ be skew-adjoint operators on a Hilbert space $H$. 
\begin{itemize}
    \item[(i)] If $a+b$ is essentially skew-adjoint on ${\rm dom}(a)\cap {\rm dom}(b)$, then it holds that
    \begin{gather*}
        e^{t\overline{(a+b)}} = s-\lim\limits_{n\rightarrow \infty}(e^{ta/n}e^{tb/n})^{n},
    \end{gather*}
    \noindent for all $t \in \bb{R}$.
    \item[(ii)] If $(ab-ba)$ is essentially skew-adjoint on 
    \begin{gather*}
    {\rm dom}(a^{2}) \cap {\rm dom}(ab) \cap {\rm dom}(ba) \cap {\rm dom}(b^{2}),
    \end{gather*}
    \noindent then it holds that
    \begin{gather*}
        e^{t[a, b]} = s-\lim\limits_{n\rightarrow \infty}\bigg(e^{-\sqrt{t/n}a}e^{-\sqrt{t/n}b}e^{\sqrt{t/n}a}e^{\sqrt{t/n}b}\bigg)^{n^{2}},
    \end{gather*}
    \noindent for $t > 0$ where $[a, b] := \overline{ab-ba}$. 
\end{itemize}
\end{lemma}
\noindent Following the notation introduced in \cite{am_one}, for a finite von Neumann algebra $\mathcal{M}$ on $H$ if $G \subseteq U(\mathcal{M})$ is a strongly closed subgroup, the set
\begin{gather*}
    \frak{g} = {\rm Lie}(G) = \{a: a^{*} = -a \; {\rm on} \; H, e^{ta} \in G \; {\rm for \; all \;} t \in \bb{R}\}
\end{gather*}
\noindent is called the \textit{Lie algebra of $G$}; this is a well-defined Lie algebra, considered inside $L^{0}(\mathcal{M}, \tau)$ (see \cite[Lemma 4.3]{am_one}). The complexification $\frak{g}_{\bb{C}}$ is given by
\begin{gather*}
    \frak{g}_{\bb{C}} := \{\overline{a+ib}: \; a, b \in \frak{g}\}.
\end{gather*}
\noindent If $G = U(\mathcal{M})$, we denote $\frak{g} = \frak{u}(\mathcal{M})$. For $a_{1}, \hdots, a_{k} \in \frak{g}$, we let $\langle a_{1}, \hdots, a_{k}\rangle_{{\rm Lie}, \bb{R}} \subseteq \frak{g}$ denote the smallest linear subspace of $\frak{g}$ which also contains all iterated brackets $[a_{i}, a_{j}], [a_{i}, [a_{j}, a_{\ell}]]$, etc. for $i, j, \ell \in [k]$; that is, the Lie subalgebra generated by $a_{1}, \hdots, a_{k}$.


\section{Admissible controls for quantum systems}\label{sec:adm_controls}
In Sections \ref{sec:adm_controls} and \ref{sec:qcs_geom_control}, we consider infinite-dimensional, closed quantum systems evolving according to a Hamiltonian $H(t)$ which is a solution to the time-dependent Schrödinger equation; assume further that $H(t)$ is of the form $H(t) = V_{0}+V(t)$, where $V_{0}$ is the time-independent Hamiltonian of the system when isolated, and $V(t)$ describes the interaction of the system with a control. Recall (see Section \ref{ss:intro_qcs}) that if $H$ denotes the separable, infinite-dimensional Hilbert space for the system and $\xi = \xi(t) \in H$ is the time-dependent vector state , then $\xi$ satisfies
\begin{gather}\label{eqn:schrodinger}
    i\hbar\frac{d}{dt}\xi(t) = (V_{0}+V(t))\xi(t), \;\;\;\; \xi(0) = \xi_{0}.
\end{gather}
While Hamiltonians are often denoted in the literature by $H$ or $\hat{H}$ (potentially with some subscript), we break with tradition in the hopes of avoiding confusion when dealing with operators on Hilbert space $H$. Here and in the sequel, we fix a finite von Neumann algebra $(\mathcal{M}, \tau)$ acting on $H$. By separability of $H$, the predual $\mathcal{M}_{\ast}$ of $\mathcal{M}$ is separable as well. For $t \geq 0$, we let $U^{t} := e^{-i/\hbar t V_{0}}$ denote the unitary on $H$ generated by $V_{0}$. Additionally, we fix $T > 0$ and $N > 0$. We endow $[0, T]$ with the (normalized) Lebesgue measure $\lambda$, making $([0, T], \Sigma_{\lambda}, \lambda)$ into a finite measure space. For the following definitions, we follow conventions similar to those in \cite{hls}. 

\begin{definition}
A {\rm control function on} $[0, T]$ is a self-adjoint operator $V(t)$ defined over all $0 \leq t \leq T$ acting on $H$ such that $V(t) \in L^{0}(\mathcal{M}, \tau)$ for all $t \in [0, T]$.
\end{definition}
\begin{definition}
A control function is {\rm admissible with bound $N$} if $V(t) \in \mathcal{M}$ for $t \in [0, T]$, and moreover the mapping $t \mapsto V(t)$ is:
\begin{itemize}
    \item[(i)] measureable;
    \item[(ii)] strongly continuous on $H$;
    \item[(iii)] satisfies $\|V(t)\| \leq N$, for $t \in [0, T]$.
\end{itemize}
\noindent We denote by $\mathcal{E}^{N}[0, T]$ the set of admissible control functions over $[0, T]$ with cutoff $N$.
\end{definition}

\begin{remark}
\rm Under the previous assumptions, it is clear that a sufficient condition for $V(t)$ to be admissible with bound $N$ is if the mapping $t \mapsto V(t)$ is operator norm continuous over $[0, T]$. 
\end{remark}

\noindent Note that $\mathcal{E}^{N}[0, T]$ is contained inside the Banach space $L_{\sigma}^{\infty}([0, T], \mathcal{M})$. By separability of the predual $\mathcal{M}_{\ast} \cong L^{1}(\mathcal{M}, \tau)$, we have that $\mathcal{M}$ is separable by Lemma \ref{lem:sep_of_lp}. Additionally, one has
\begin{gather*}
    L_{\sigma}^{\infty}([0, T], \mathcal{M}) \cong L^{1}([0, T], L^{1}(\mathcal{M}, \tau))^{*} \cong (L^{1}([0, T], \lambda)\hat{\otimes}L^{1}(\mathcal{M}, \tau))^{*},
\end{gather*}
\noindent where $\hat{\otimes}$ denotes the projective tensor product of Banach spaces. Thus, we endow $\mathcal{E}^{N}[0, T]$ with a ${\rm wk}^{*}$-topology, inherited from $L_{\sigma}^{\infty}([0, T], \mathcal{M})$. Under the previous assumptions on $H(t)$, and with the further assumption we work with the inhomogeneous Schrödinger equation (which can be viewed as a first-order Born approximation) of the form
\begin{gather}\label{eqn:approx_schrodinger}
    i\hbar\frac{d}{dt}\xi(t) = V_{0}\xi(t)+V(t)U^{t}\xi_{0}, \;\;\;\; \xi_{0} \in H,
\end{gather}
\noindent if $\xi(t)$ satisfies (\ref{eqn:approx_schrodinger}) standard results (see \cite[Chapter Nine $\S$4]{kato}) show that the explicit form is
\begin{gather}\label{eqn_solution_for_vector}
    \xi(t) = U^{t}\xi_{0}+\int\limits_{0}^{t}U^{t-s}V(s)U^{s}\xi_{0}ds, \;\;\;\; t \geq 0.
\end{gather}
Conversely, if the mapping $t \mapsto V(t)U^{t}\xi_{0}$ is strongly differentiable with strongly-continuous derivative (and $V \in \mathcal{E}^{N}[0, T])$, then (\ref{eqn:schrodinger}) has a unique solution $\xi(t) \in H$ for each $t \geq 0$. We specifically note that as $V_{0} \in L^{0}(\cl{M}, \tau)$, then $U^{t} \in \mathcal{M}$ for all $t \in [0, T]$.

For $V = V(t) \in \mathcal{E}^{N}[0, T]$, the vector state at time $T$ is the corresponding value $\xi(T) \in H$ under the dynamics governed by $V_{0}$ and $V(t)$. Let $\Psi: \mathcal{E}^{N}[0, T]\rightarrow H$ be the map sending $V \in \mathcal{E}^{N}[0, T]$ to the integral term in $\xi(T)$ of $H$ as given in (\ref{eqn_solution_for_vector}).
\begin{definition}
The {\rm reachable set at time $T$}, for parameter $N$, is denoted by $\mathcal{R}^{N}(T)$ and is given by $\mathcal{R}^{N}(T) = \Psi(\mathcal{E}^{N}[0, T])$. We say a vector state $\xi \in H$ is {\rm reachable at time $T$}, for some parameter $N$ if there exists some $\xi_{T} \in \mathcal{R}^{N}(T)$ such that $\xi = U^{T}\xi_{0}+\xi_{T}$.
\end{definition}
\begin{proposition}\label{prop:properties_of_reachable_set}
For each $N >0$, the set $\mathcal{R}^{N}(T)$ is weakly compact and convex. Furthermore, if $T_{1} \leq T_{2}$, then $\mathcal{R}^{N}(T_{1}) \subseteq \mathcal{R}^{N}(T_{2})$. 
\end{proposition}
\begin{proof}
Convexity is easily seen: assume $V, V' \in \mathcal{E}^{N}[0, T]$ and pick $0 \leq \theta \leq 1$. Then
\begin{gather*}
    \|\theta V(t)+(1-\theta)V'(t)\| \leq \theta\|V(t)\|+(1-\theta)\|V'(t)\| \leq N,
\end{gather*}
\noindent for all $0 \leq t \leq T$, showing $\mathcal{E}^{N}[0, T]$ is convex. By convex-linearity of $\Psi$ (where we consider convex combinations of uniformly bounded elements taken inside $L_{\sigma}^{\infty}([0, T], \mathcal{M})$), and as $\mathcal{R}^{N}(T) = \Psi(\mathcal{E}^{N}[0, T])$ is the image of a convex set under a convex-linear map, it must be convex as well.

The inclusion $\mathcal{R}^{N}(T_{1}) \subseteq \mathcal{R}^{N}(T_{2})$ for $T_{1} \leq T_{2}$ is also easily seen: if $V \in \mathcal{E}^{N}[0, T_{1}]$, define
\begin{gather*}
    \widehat{V}(t) := \begin{cases}
    V(t), \;\;\;\; 0 \leq t \leq T_{1},
        \\ 0, \;\;\;\;\;\;\;\;\;\;\; T_{1} < t \leq T_{2}.
    \end{cases}
\end{gather*}
\noindent It is clear $\hat{V} \in \mathcal{E}^{N}[0, T_{2}]$, which is sufficient to show the claimed inclusion.

To show compactness, first note that
\begin{gather*}
    \mathcal{E}^{N}[0, T] = N\cdot (L_{\sigma}^{\infty}([0, T], \mathcal{M}))_{1},
\end{gather*}
\noindent where the latter denotes the unit-ball. By the Banach-Alaoglu Theorem, it is ${\rm wk}^{*}$-compact (and therefore ${\rm wk}^{*}$-closed as well). To finish the proof, we verify that $\Psi$ is ${\rm wk}^{*}$-continuous. 

To that end, assume that $V_{n} \rightarrow_{\rm wk^{*}} V$ as $n\rightarrow \infty$ in $L_{\sigma}^{\infty}([0, T], \mathcal{M})$; then by the previous comments, for any $F \in L^{1}([0, T], L^{1}(\mathcal{M}, \tau))$ we have
\begin{gather*}
    \langle V_{n}, F\rangle = \int\limits_{0}^{T}\tau(V_{n}(t)F(t))dt\rightarrow \int\limits_{0}^{T}\tau(V(t)F(t))dt = \langle V, F\rangle,
\end{gather*}
\noindent where $V_{n}(t), V(t) \in M, F(t) \in L^{1}(\mathcal{M}, \tau)$ for all $t \in [0, T]$.  We will use this characterization of convergence to show that for any $\eta \in H$,
\begin{gather}\label{eqn:conv_of_phi_image}
    \bigg\langle \int\limits_{0}^{T}U^{-t}V_{n}(t)U^{t}\xi_{0}dt, \eta\bigg\rangle\rightarrow \bigg\langle \int\limits_{0}^{T}U^{-t}V(t)U^{t}\xi_{0}dt, \eta\bigg\rangle,  
\end{gather}
\noindent as $n\rightarrow \infty$.

Admissibility conditions on $V_{n}$ for each $n \in \mathbb{N}$ ensure that the integral term in (\ref{eqn_solution_for_vector}) exists in the strong sense, even by removing a factor of $U^{T}$ for fixed $T$. This also ensures that the integral is well-defined in the weak sense (i.e., as a Dunford-Pettis integral by the reflexivity of $H$). For $\xi, \eta \in H$ and each time $t \in [0, T]$ the mapping $a \mapsto \langle aU^{t}\xi, U^{t}\eta\rangle$ is a normal functional on $\mathcal{M}$. By the correspondence between $\mathcal{M}_{\ast}$ and $L^{1}(\mathcal{M}, \tau)$, this implies there exists some $F_{\xi, \eta}(t) \in L^{1}(\mathcal{M}, \tau)$ such that $\langle aU^{t}\xi, U^{t}\eta\rangle = \tau(a F_{\xi, \eta}(t))$ for all $a \in \mathcal{M}$. Let $F_{\xi, \eta}: [0, T]\rightarrow L^{1}(\mathcal{M}, \tau)$ denote the mapping $t \mapsto F_{\xi, \eta}(t)$ for fixed $\xi, \eta \in H$.

By properties of the weak integral and the inner product on $H$, we have that
\begin{eqnarray*}
    \bigg\langle\int\limits_{0}^{T}U^{-t}V_{n}(t)U^{t}\xi_{0}dt, \eta\bigg\rangle 
    & = &
    \int\limits_{0}^{T}\langle U^{-t}V_{n}(t)U^{t}\xi_{0}, \eta\rangle dt \\
    & = &
    \int\limits_{0}^{T}\langle V_{n}(t)U^{t}\xi_{0}, U^{t}\eta\rangle dt \\
    & = &
    \int\limits_{0}^{T}\tau(V_{n}(t)F_{\xi_{0}, \eta}(t))dt,
\end{eqnarray*}
\noindent where $F_{\xi_{0}, \eta}(t) \in L^{1}(\mathcal{M}, \tau)$ for all $t \in [0, T]$ by previous remarks. Thus,
\begin{gather*}
    \int\limits_{0}^{T}\tau(V_{n}(t)F_{\xi_{0}, \eta}(t))dt\rightarrow \int\limits_{0}^{T}\tau(V(t)F_{\xi_{0}, \eta}(t))dt
\end{gather*}
\noindent as $n\rightarrow \infty$, which (by a reversal of the arguments above) establishes (\ref{eqn:conv_of_phi_image}). As multiplication by fixed unitary $U^{T}$ is a bounded (norm-continuous) linear mapping, we may bring $U^{T}$ inside the integral so that
\begin{gather*}
    \bigg\langle\int\limits_{0}^{T}U^{T-t}V_{n}(t)U^{t}\xi_{0}dt, \eta\bigg\rangle \rightarrow \bigg\langle \int\limits_{0}^{T}U^{T-t}V(t)U^{t}\xi_{0}dt, \eta\bigg\rangle, \;\;\;\; n\rightarrow \infty.
\end{gather*}
By the reflexivity of $H$ as a Hilbert space, this shows $\Psi(V_{n})\rightarrow \Psi(V)$ weakly (and thus in the ${\rm wk}^{*}$-topology), and so $\Psi$ is ${\rm wk}^{*}$-continuous as claimed. This concludes the proof. 
\end{proof} 
\begin{remark}
\rm It is easily verified that if $N, N' > 0$ with $N' \geq N$, then $\mathcal{R}^{N}(T) \subseteq \mathcal{R}^{N'}(T)$. For $T_{1}, T_{2} > 0$ with $T_{1} \leq T_{2}$, this yields the inclusions
\begin{gather*}
\renewcommand{\arraystretch}{0}
\begin{matrix}
\mathcal{R}^{N}(T_{2}) & \subset & \mathcal{R}^{N'}(T_{2}) \\[2ex]
\rotatebox[origin=c]{90}{$\subset$} &&
\rotatebox[origin=c]{90}{$\subset$} \\[2.5ex]
\mathcal{R}^{N}(T_{1}) & \subset & \mathcal{R}^{N'}(T_{1})
\end{matrix}
\end{gather*}
\end{remark}

\begin{theorem}
Suppose the system reaches some vector state $\xi \in H$ at time $T > 0$ for parameter $N > 0$; then there exists a admissible control function for which the system reaches $\xi$ in minimal time.
\end{theorem}
\begin{proof}
The proof follows in a manner similar to standard Filippov existence-type theorems (see for instance \cite[Theorem 13.1]{hls}). Let
\begin{gather*}
    T^{*} = \inf\{T: \xi_{T} \in \mathcal{R}^{N}(T) \; {\rm such \; that \;} \xi = U^{T}\xi_{0}+\xi_{T}\};
\end{gather*}
\noindent then there exists a sequence $(T_{n})_{n=1}^{\infty}$ decreasing to $T^{*}$ with $\xi_{T_{n}} \in \mathcal{R}^{N}(T_{n})$ for each $n \in \mathbb{N}$. Suppose $V_{n} \in \mathcal{E}^{N}[0, T_{n}]$ is the control function corresponding to $\xi_{T_{n}}$ for time $T_{n}$ (under the mapping $\Psi$). For $t > T_{n}$, set $V_{n}(t) = 0$; abusing notation slightly, call the resulting control function $V_{n}$. Note that by Proposition \ref{prop:properties_of_reachable_set}, as $T_{1} \geq T_{2} \geq \cdots \geq T^{*}$ then $\mathcal{E}^{N}[0, T_{i}] \subseteq \mathcal{E}^{N}[0, T_{1}]$ for all $i \geq 1$; thus, we may take a ${\rm wk}^{*}$-limit point $V$ of $(V_{n})_{n=1}^{\infty}$ inside $\mathcal{E}^{N}[0, T_{1}]$. By ${\rm wk}^{*}$-continuity of $\Psi$, we claim that $\Psi(V)$ reaches $\xi$ at time $T^{*}$. More explicitly, note that for a fixed self-adjoint $V \in L^{0}(\mathcal{M}, \tau)$ defined over a time interval with $\|V\|\leq N$, then if $T, T' \geq 0$ and if $V$ is defined for parameters $T, T'$ we have 
\begin{eqnarray*}
    \bigg\|\int\limits_{0}^{T}U^{-t}V(t)U^{t}\xi_{0} dt-\int\limits_{0}^{T'}U^{-t}V(t)U^{t}\xi_{0} dt\bigg\| 
    & = & \bigg\|\int\limits_{T'}^{T}U^{-t}V(t)U^{t}\xi_{0}dt\bigg\| \\
    &\leq & N|T-T'|.
\end{eqnarray*}
\noindent Pairing this with the ${\rm wk}^{*}$-continuity of $\Psi$, we conclude the proof.
\end{proof} 


\section{Quantum control systems and the dynamical Lie algebra}\label{sec:qcs_geom_control}

In this section, we specialize to the case where the control Hamiltonian $V(t)$ in (\ref{eqn:schrodinger}) further decomposes as a finite sum of Hamiltonians with piecewise constant, bounded control functions over all time intervals in order to pursue the following main objectives. First, we show how under such assumptions using the results of \cite{am_one} we may construct Lie algebras in $L^{0}(\mathcal{M}, \tau)$ using the Hamiltonians, recovering the ``dynamical Lie algebra" often discussed in geometric control theory (see \cite{daless, js}). Second, we discuss various notions of controllability for quantum systems (without further discussion on time-optimal control), and indicate distinctions between the finite and infinite-dimensional setting. Finally, we propose an operator-algebraic informed approximation scheme for the time-independent Hamiltonian when it possesses complicated spectral behavior.

\subsection{The dynamical Lie algebra}

Assume that the Hamiltonian $H(t)$ arising in (\ref{eqn:schrodinger}) is of the form 
\begin{gather}\label{eqn_bilinear_system}
    H(t) = V_{0}+\sum\limits_{j=1}^{N}u_{j}(t)V_{j},
\end{gather}
\noindent where $V_{j}, j = 0, \hdots, N$ are (potentially unbounded) self-adjoint operators on $H$, and $u_{j}: \bb{R}\rightarrow \bb{R}$ are piecewise constant control functions for $j = 1, \hdots, N$; systems of this form are sometimes called ``control-affine" in the literature. Now and in the sequel, we drop any other assumptions on solutions (i.e., the assumptions in (\ref{eqn:approx_schrodinger}) are no longer required). 
\begin{lemma}
Let $\mathcal{M}$ be a finite von Neumann algebra, and assume $V_{j} \in L^{0}(\mathcal{M}, \tau)$ are self-adjoint for $j = 0, \hdots, N$. Then the $V_{j}$'s have a common dense domain in $H$, and 
\begin{gather}\label{eqn:v_y}
    V(y) := V_{0}+\sum\limits_{j=1}^{N}y_{j}V_{j} \; \eta \; \mathcal{M},
\end{gather}
\noindent with $V(y)$ self-adjoint for every choice of $y = (y_{1}, \hdots, y_{N}) \in \bb{R}^{N}$.
\end{lemma}
\begin{proof}
That $V_{j}, j= 0, \hdots, N$ have a common dense domain follows from \cite[Lemma 16.4.3]{mvn}; indeed, any polynomial with variables over $a_{i}, a_{i}^{*}$ when $a_{i}, a_{i}^{*}$ are affiliated with $\mathcal{M}$ have a dense domain in $H$. By virtue of \cite[Theorem XV]{mvn}, each $V(y) \; \eta \; \mathcal{M}$ for $y \in \bb{R}^{N}$, and the fact that $V(y)$ is self-adjoint follows from the fact that each $V_{j}$ is, in conjunction with \cite[Lemma 16.4.2]{mvn}.
\end{proof}
\begin{remark}
\rm In the sequel, for $y \in \mathbb{R}^{N}$ we let $V(y)$ denote the linear combination of $V_{0}, \hdots, V_{N}$ using $y$ as defined in (\ref{eqn:v_y}).
\end{remark}
We note that if $V(t)$ decomposes as in (\ref{eqn_bilinear_system}), then $V(t)$ is automatically a control function over $[0, T]$ for any positive time $T > 0$. If we assume $V_{1}, \hdots, V_{N}$ have some joint regularity (that is, $V_{j} \in \cl{M}$ for $j = 1, \hdots, N$) then $V(t)$ is automatically an admissible control function for some cutoff $N'$, where $N'$ comes from the maximal bound on the values of the piece-wise constant functions $u_{j}$, with $j = 1, \hdots, N$ and the maximum norm bound over the Hamiltonians. However, we will not require any such regularity for the results in this section. Indeed: as we will see, we may still discuss various algebraic notions of controllability for the system without such assumptions, and analytic concerns will be less relevant to our approach. 

\begin{proposition}
If $V_{0}, \hdots, V_{N}$ are (potentially unbounded) self-adjoint operators acting on separable Hilbert space $H$, there exists at least one von Neumann algebra $\mathcal{M}$ such that $V_{j} \in L^{0}(\mathcal{M})$ for $j = 0, \hdots, N$. 
\end{proposition}
\begin{proof}
That each $V_{j}$ is affiliated with an abelian von Neumann algebra $\mathcal{M}_{j}$ for $j = 1, \hdots, N$ acting on $H$ follows from \cite[Lemma 5.6.7]{kr_one}. Let $\mathcal{M}$ be the von Neumann algebra on $H$ generated by all $\mathcal{M}_{j}$'s; that is, let $\mathcal{M}$ be the von Neumann algebra generated by each $\mathcal{M}_{j}$ (considered as a union in the set-theoretic sense). We claim that $V_{j} \in L^{0}(\mathcal{M})$. Indeed, as $\mathcal{M}_{j} \subseteq \mathcal{M}$, then $\cl{M}' \subseteq \mathcal{M}_{j}'$ for $j = 0, \hdots, N$. Thus, if $u \in U(\cl{M}')$, $u \in U(\cl{M}_{j}')$ for $j = 0, \hdots, N$, and so $uV_{j}u^{*} = V_{j}$ for $j = 0, \hdots, N$. This means $V_{j} \in L^{0}(\mathcal{M})$ for each $j = 0, \hdots, N$. 
\end{proof}
\begin{remark}
\rm In order to guarantee $\mathcal{M}$ is finite type, in general more conditions on the $V_{j}$'s are be needed. One sufficient condition is if the spectrum of the inverses of the $(V_{j}+iI_{H})$'s are pairwise disjoint. Under this assumption, \cite[Lemma 3]{wogen} implies $\mathcal{M}$ is a finite direct sum of abelian von Neumann algebras, and hence of finite type.
\end{remark}

\medskip

\begin{definition}
Let $H$ be a Hilbert space, and $\mathcal{M}$ a finite von Neumann algebra acting on $H$; furthermore, let $V_{0}, \hdots, V_{N}$ be affiliated with $\mathcal{M}$. We define
\begin{itemize}
    \item[(i)] $\cl{R}_{0}(V_{0}, \hdots, V_{N})$ as the smallest sub-semigroup of $U(\mathcal{M})$ containing all unitaries $e^{itV(y)}$ with $t > 0$ and $y \in \bb{R}^{N}$; we call it the {\rm algebraically reachable set}. 
    \item[(ii)] $\cl{R}(V_{0}, \hdots, V_{N})$ to be the closure of $\cl{R}_{0}(V_{0}, \hdots, V_{N})$ under the strong operator topology inside $U(\mathcal{M})$; we call it the {\rm strongly reachable set}. 
    \item[(iii)] $\cl{G}(V_{0}, \hdots, V_{N})$ as the smallest, strongly closed subgroup of $U(\mathcal{M})$ containing all unitaries $e^{it V(y)}$ with $t \in \bb{R}$ and $y \in \bb{R}^{N}$. 
\end{itemize}
\end{definition}
\begin{remark}
\rm Note that $\mathcal{G}(V_{0}, \hdots, V_{N})$ may equivalently be characterized as the smallest, strongly closed subgroup of $U(\mathcal{M})$ containing all unitaries $e^{it V_{j}}$, where $t \in \mathbb{R}$ and $j = 0, \hdots, N$. This follows from Remark \ref{rem:completely_dense} and an application of \cite[Proposition 2.2]{keyl}. In the sequel, we may use this identification without further mention.
\end{remark}
\begin{proposition}\label{p_strong_reach_equals_group}
Suppose $V_{0}, \hdots, V_{N} \in L^{0}(\mathcal{M}, \tau)$, and $V_{0}$ has pure-point spectrum. Then
\begin{gather*}
    \cl{R}(V_{0}, \hdots, V_{N}) = \cl{G}(V_{0}, \hdots, V_{N}).
\end{gather*}
\end{proposition}
\begin{proof}
The ideas are similar to those in \cite{keyl}. By definition, it is clear that $\cl{R}(V_{0}, \hdots, V_{N}) \subseteq \cl{G}(V_{0}, \hdots, V_{N})$. We wish to show that $e^{itV(y)} \in \cl{R}(V_{0}, \hdots, V_{N})$ for all $t \leq 0$. In the case when $t = 0$, by strong continuity we have $I \in \mathcal{R}(V_{0}, \hdots, V_{N})$ as we may view $I$ as the strong limit of $e^{itV(y)} \in \mathcal{R}(V_{0}, \hdots, V_{N})$ when $t\rightarrow 0^{+}$.

If $t < 0$, choose $s > 0$ with $t' := t-s$. Through some basic algebra, we have $tV(y) = t' V_{0}+sV\bigg(\frac{t}{s}y\bigg)$. As $V_{0}$ has pure-point spectrum, let $(\psi_{n})_{n \in \mathbb{N}}$ denote an orthonormal basis of $H$ consisting of eigenvectors for $V_{0}$. Let $p_{N}$ denote the orthogonal projection of $H$ onto the span of $\{\psi_{1}, \hdots, \psi_{N}\}$; by invariance of $p_{N}H$ under $V_{0}$, the restriction $e^{it'V_{0}}|_{p_{N}H}$ defines a one-parameter group of unitaries on $p_{N}H$. As this space is finite-dimensional, for any $\epsilon > 0$ there exists a $t_{+}' > 0$ such that $\|e^{it_{+}'V_{0}}|_{p_{N}H}-e^{it'V_{0}}|_{p_{N}H}\| < \epsilon/3$. Therefore, for any $\epsilon > 0$ and unit vector $\psi \in H$ take $N \in \mathbb{N}$ such that $\|p_{N}^{\perp}\psi\| <\epsilon/3$. Now, 
\begin{eqnarray*}
    \|e^{it_{+}'V_{0}}\psi-e^{itV_{0}}\psi\| 
    & \leq &
    \|(e^{it_{+}'V_{0}}-e^{it'V_{0}})p_{N}\psi\|+\|(e^{it_{+}'V_{0}}-e^{it'V_{0}})p_{N}^{\perp}\psi\|\\
    & \leq &
    \frac{\epsilon}{3}+\|e^{it_{+}'V_{0}}\|\|p_{N}^{\perp}\psi\|+\|e^{it'V_{0}}\|\|p_{N}^{\perp}\psi\| \\
    & \leq &
    \epsilon.
\end{eqnarray*}
\noindent As $\epsilon > 0$ was arbitrary, this implies that for any $\epsilon > 0$ there exists some $t_{+}' > 0$ with $e^{it_{+}'V_{0}}$ contained in a strong neighborhood of radius $\epsilon $ around $e^{it'V_{0}}$. As $t_{+}'V_{0} \; \eta \; \mathcal{M}$ whenever $V_{0} \; \eta \; \mathcal{M}$, this implies $e^{it'V_{0}} \in \cl{R}(V_{0}, \hdots, V_{N})$. 

By definition of $\cl{R}(V_{0}, \hdots, V_{N})$, we have $e^{isV(\frac{t}{s}y)} \in \cl{R}(V_{0}, \hdots, V_{N})$ as well. Using Lemma \ref{l_tk_nelson} (i), all of these together imply $e^{itV(y)} \in \cl{R}(V_{0}, \hdots, V_{N})$. By strong closure, this shows $e^{itV(y)} \in \cl{R}(V_{0}, \hdots, V_{N})$ for arbitrary $t \in \bb{R}$, and $y \in \bb{R}^{N}$; thus, $\cl{R}(V_{0}, \hdots, V_{N}) = \cl{G}(V_{0}, \hdots, V_{N})$.

\end{proof}
\begin{lemma}
Let $H$ be a Hilbert space, $\mathcal{M}$ a finite von Neumann algebra acting on $H$, and let $V_{0}, \hdots, V_{N}$ be self-adjoint operators on $H$ affiliated with $\mathcal{M}$. Then $\cl{R}(V_{0}, \hdots, V_{N})$ is a connected semigroup. Furthermore, if $V_{0}$ has pure-point spectrum then $\cl{G}(V_{0}, \hdots, V_{N})$ is a connected subgroup of $U(\mathcal{M})$.
\end{lemma}
\begin{proof}
By \cite[Lemma 4.4]{js}, $\cl{R}_{0}(V_{0}, \hdots, V_{N})$ is path connected, and hence connected. As the closure (in this case, inside the strong operator topology) of a connected set remains connected, we have that $\cl{R}(V_{0}, \hdots, V_{N})$ is connected. To see that $\cl{R}(V_{0}, \hdots, V_{N})$ is a semigroup, first note that multiplication by unitaries is separately continuous in the strong operator topology. This makes $\mathcal{R}_{0}(V_{0}, \hdots, V_{N})$ a (left) topological sub-semigroup inside $\mathcal{G}(V_{0}, \hdots, V_{N})$. Now, apply \cite[Proposition 1.4.10]{at} to conclude that $\mathcal{R}(V_{0}, \hdots, V_{N})$ is a semigroup. Finally, if we assume $V_{0}$ has pure-point spectrum, applying Proposition \ref{p_strong_reach_equals_group} we obtain $\cl{R}(V_{0}, \hdots, V_{N}) = \cl{G}(V_{0}, \hdots, V_{N})$ and thus the latter is a connected group. 
\end{proof}

We wish to discuss a few notions of controllability for systems of the form (\ref{eqn:schrodinger}), and how this connects to the Lie algebraic picture. We say a system is \textit{pure state controllable} if for any two unit vectors $\xi_{0}, \xi_{f} \in H$, there exists a control function and a time $T > 0$ such that the solution of (\ref{eqn:schrodinger}) with initial condition $\xi_{0}$, is $\xi(T) = \xi_{f}$. For systems which further decompose in the form (\ref{eqn_bilinear_system}), the following definition is after \cite{bms}.
\begin{definition}
Let $\mathcal{M}$ be a finite von Neumann algebra, and assume the Hamiltonians arising from (\ref{eqn_bilinear_system}) satisfy $V_{j} \in L^{0}(\mathcal{M}, \tau)$ for $j = 0, \hdots, N$. Then the system (\ref{eqn_bilinear_system}) is called strongly $\mathcal{M}$-controllable if $\cl{R}(V_{0}, \hdots, V_{N}) = U(\mathcal{M})$. 
\end{definition}

\begin{lemma}\label{lemma:factor_transitivity}
If $\mathcal{M}$ is not a factor, then $U(\mathcal{M})$ does not act transitively on the set of unit vectors in $H$. 
\end{lemma}
\begin{proof}
Let $\mathcal{Z}(\mathcal{M}) = \mathcal{M}\cap \mathcal{M}'$ denote the center of $\mathcal{M}$. With $\mathcal{M}$ faithfully represented on $H$, take a non-trivial projection $p \in \mathcal{P}(\mathcal{Z}(\mathcal{M}))$. Then for any $u \in U(\mathcal{M})$, $up = pu$, and so $pH$ and $p^{\perp}H$ are invariant subspaces. If we take $\xi \in pH, \xi' \in p^{\perp}H$ unit vectors, there does not exist a $u \in U(\mathcal{M})$ such that $u\xi = \xi'$. 
\end{proof}

\begin{corollary}
If $\mathcal{M}$ is not a factor, then the system in (\ref{eqn_bilinear_system}) is not pure state controllable. 
\end{corollary}
\begin{proof}
As any solution to (\ref{eqn:schrodinger}) with initial condition $\xi_{0}$ and Hamiltonian of the form (\ref{eqn_bilinear_system}) can be written as $\xi(t) = u(t)\xi_{0}$ with $u(t) \in U(\mathcal{M})$ over some time interval $[0, T]$, even if $\mathcal{R}(V_{0}, \hdots, V_{N}) = U(\mathcal{M})$ by Lemma \ref{lemma:factor_transitivity} $\mathcal{R}(V_{0}, \hdots, V_{N})$ does not act transitively on the unit sphere of $H$. Hence, the system is not pure-state controllable.
\end{proof}

\begin{remark}\label{remark:controllability_differences}
\rm In order to better understand potential differences for control between the infinite-dimensional and the finite-dimensional settings, we briefly recall analogues of controllability for the latter; proofs may be found in \cite[Chapter 3]{daless}. 

If our Hilbert space $H$ is finite-dimensional, then $H \cong \mathbb{C}^{n}$ for some $n \in \mathbb{N}$. All Hamiltonians acting on $H$ are now automatically bounded and in $M_{n}(\mathbb{C})$. Here, there is (essentially) only one choice of von Neumann algebra with $\mathcal{M}= M_{n}(\mathbb{C})$, and by boundedness the affiliation relation trivializes. This is a finite type ${\rm I}_{n}$ factor, with the standard normalized matrix trace ${\rm Tr}: M_{n}(\mathbb{C})\rightarrow \mathbb{C}$. The definition of pure state controllability does not change; as the strong and norm operator topologies coincide on $M_{n}(\mathbb{C})$, strong $\mathcal{M}$-controllability reduces to \textit{operator controllability}, where $\mathcal{R}(V_{0}, \hdots, V_{N}) = U(n)$ (or even $SU(n)$, depending on the dimension of the dynamical Lie algebra). Furthermore, if we assume the system is an ensemble and the dynamics is described by a Liouville-von Neumann equation $i\hbar\frac{d}{dt}\rho(t) = [V(t), \rho(t)]$ for the density matrix $\rho(t)$, then for any two unitarily equivalent density matrices $\rho_{0}, \rho_{f}$ we say that the system is \textit{density matrix controllable} if there exists a control $V(t)$ and some time $T > 0$ such that the solution $\rho(t)$ to the dynamics with initial condition $\rho_{0}$ satisfies $\rho(T) = \rho_{f}$. As one may show, density matrix controllability is equivalent to operator controllability (strong $\mathcal{M}$-controllability in this setting), and these imply pure state controllability. 

To contrast with our results in the infinite-dimensional setting, we already see how the analogue of operator controllability no longer implies pure state controllability in most cases. Additionally, we would not expect strong $\mathcal{M}$-controllability to coincide with controllability between arbitrary densities $\rho_{0}, \rho_{f} \in \mathcal{T}(H)^{+}$ (or even $L^{1}(\mathcal{M}, \tau)^{+}$); this would be asking for arbitrary density conversion using a select subset of unitarily implementable operations, coming from a subalgebra $\mathcal{M}\subseteq \mathcal{B}(H)$.
\end{remark}

While the previous remark shows relations between various notions of controllability no longer hold in the same way as in the finite-dimensional setting, under minor assumptions the following shows we can recover a version of the Lie Algebra Rank Condition (see \cite[Theorem 3.2.1]{daless}). Recall that, as $\mathcal{G}(V_{0}, \hdots, V_{N}) \subseteq U(\mathcal{M})$ is a strongly closed subgroup, corresponding Lie algebra $\mathfrak{g}(V_{0}, \hdots, V_{N})$ is well-defined inside $L^{0}(\mathcal{M}, \tau)$ (see Section \ref{sec:preliminaries}). 

\begin{theorem}\label{t_larc}
Suppose $V_{0}, \hdots, V_{N} \in L^{0}(\mathcal{M}, \tau)$. Then
\begin{gather*}
    \langle iV_{0}, \hdots, iV_{N}\rangle_{{\rm Lie}, \bb{R}} = \mathfrak{g}(V_{0}, \hdots, V_{N}),
\end{gather*}
\noindent where the former is considered as a real, SRT-closed Lie subalgebra inside $L^{0}(\mathcal{M}, \tau)$.  
\end{theorem}
\begin{proof}
Note that as $\mathcal{G}(V_{0}, \hdots, V_{N})$ is the smallest strongly closed subgroup of $U(\mathcal{M})$ containing all elements of the form $e^{it V(y)}$ with $t \in \mathbb{R}$ and $y \in \mathbb{R}^{N}$, and as $itV(y) \in \langle iV_{0}, \hdots, iV_{N}\rangle_{{\rm Lie}, \mathbb{R}}$ for all $t \in \mathbb{R}, y \in \mathbb{R}^{N}$, we have the inclusion $\mathcal{G}(V_{0}, \hdots, V_{N})$ $\subseteq e^{\langle iV_{0}, \hdots, iV_{N}\rangle_{{\rm Lie}, \mathbb{R}}}$, which implies $\mathfrak{g}(V_{0}, \hdots, V_{N}) \subseteq \langle iV_{0}, \hdots, iV_{N}\rangle_{{\rm Lie}, \mathbb{R}}$. 

For the other containment, pick $y_{1}, \hdots, y_{N} \in \mathbb{R}$ such that $y_{j} = \pm 1$ for all $j = 1, \hdots, N$ and call the corresponding vector $\hat{y} = (y_{1}, \hdots, y_{N}) \in \mathbb{R}^{N}$. Clearly, $e^{itV(\hat{y})} \in \mathcal{G}(V_{0}, \hdots, V_{N})$ for all $t \in \mathbb{R}$, with $iV(\hat{y})$ skew-adjoint. By definition, this means $iV(\hat{y}) \in \mathfrak{g}(V_{0}, \hdots, V_{N})$ for all such $\hat{y} \in \mathbb{R}^{N}$. As elements of the form $iV(\hat{y})$ as previously described generate $\langle iV_{0}, \hdots, iV_{N}\rangle_{{\rm Lie}, \mathbb{R}}$ inside $L^{0}(\mathcal{M}, \tau)$, this shows $\langle iV_{0}, \hdots, iV_{N}\rangle_{{\rm Lie}, \mathbb{R}} \subseteq \mathfrak{g}(V_{0}, \hdots, V_{N})$, completing the proof. 
\end{proof}
\begin{remark}
\rm The latter part of the argument is adapted from the well-known result \cite[Theorem 4.6]{js}, where we view the decomposition in (\ref{eqn_bilinear_system}) as the analogue of a right-invariant system over Lie group $U(\mathcal{M})$. 
\end{remark}

\begin{corollary}
Suppose $V_{0}, \hdots, V_{N} \in L^{0}(\mathcal{M}, \tau)$ and $V_{0}$ has pure-point spectrum. Then strong $\mathcal{M}$-controllability is equivalent to $\langle iV_{0}, \hdots, iV_{N}\rangle_{{\rm Lie}, \mathbb{R}} = \mathfrak{u}(\mathcal{M})$. 
\end{corollary}
\begin{proof}
By Proposition \ref{p_strong_reach_equals_group}, $\mathcal{R}(V_{0}, \hdots, V_{N}) = \mathcal{G}(V_{0}, \hdots, V_{N})$. The strong $\mathcal{M}$-controllability of the system means $\mathcal{G}(V_{0}, \hdots, V_{N}) = U(\mathcal{M})$, and so by Theorem \ref{t_larc} we have $\langle iV_{0}, \hdots, iV_{N}\rangle_{{\rm Lie}, \mathbb{R}} = \mathfrak{u}(\mathcal{M})$.

If we instead assume $\langle iV_{0}, \hdots, iV_{N}\rangle_{{\rm Lie}, \bb{R}} = \frak{u}(\mathcal{M})$, using \cite[Corollary 4.13]{am_one} we may conclude that the strongly-closed Lie subgroup corresponding to $\langle iV_{0}, \hdots, iV_{N}\rangle_{{\rm Lie}, \bb{R}}$ is $U(\mathcal{M})$ itself. Thus, $\cl{G}(V_{0}, \hdots, V_{N}) = U(\mathcal{M})$; however, another application of Proposition \ref{p_strong_reach_equals_group} shows we have strong $\mathcal{M}$-controllability.
\end{proof}

\begin{remark}
\rm  An often used tool in the geometric control theory approach to finite-dimensional quantum systems--- say, when $H$ is of dimension $n$--- is the compactness of the $n\times n$ matrix Lie group $U(n)$. In lifting to the infinite-dimensional setting, this is no longer a property we can expect. Indeed: by \cite[Lemma 2.6]{chirv}, $U(\mathcal{M})$ is compact in the strong topology if and only if $\mathcal{M}$ is hereditarily atomic: that is, $\mathcal{M}\cong \prod_{i \in I}^{W^{*}}\mathcal{B}(H_{i}),$
\noindent where $H_{i}$ is a finite-dimensional Hilbert space for $i \in I$ and the product is taken in the category of von Neumann algebras. Many interesting natural examples do not satisfy this condition even if $\mathcal{M}$ is finite: for instance, if $\mathcal{M}\cong L^{\infty}[0, 1]$, or when $\mathcal{M}\cong \mathcal{R}$, the hyperfinite ${\rm II}_{1}$ factor. 
\end{remark}

\subsection{Approximation schemes for the drift term}\label{subsec:approximation_schemes}
As previous results indicate, certain arguments become simpler (or are only possible) when the drift term $V_{0}$ has pure-point spectrum. Additionally, under this assumption the drift term is amenable to Galerkin-type approximations which aid in control of the system (see \cite{bgrs, cmsb}). However, such an assumption is not always possible to expect (see for instance \cite{bcc, lb}). In this subsection, we propose a scheme for approximating the drift term $V_{0}$ when it does not possess pure-point spectrum. Throughout, we make the further assumption that $(\mathcal{M}, \tau)$ is injective, and (abusing notation slightly) identify $V_{0} \equiv iV_{0}$, so that the drift term is now a skew-adjoint operator. Additionally, we must assume $\mathcal{M}$ is faithfully represented on $L^{2}(\mathcal{M}, \tau)$ in the standard representation; the existence of a net of projections used in the approximation is not guaranteed otherwise. 

In \cite{gv_one, gv_two}, the authors introduced a method for approximating a potentially unbounded, skew-adjoint operator which possesses mixed or continuous spectrum. In this approach, they use kernel integral operators to construct a family of skew-adjoint operators with compact resolvent (and thus, having pure-point spectrum) whose spectral measures converge in an appropriate asymptotic limit. Inspired by this approach, we introduce an alternate method for the approximation of $V_{0} \in L^{0}(\mathcal{M}, \tau)$ (in the case when it possesses mixed or continuous spectrum) by a family of skew-adjoint operators, also with compact resolvent which furthermore are affiliated with $\mathcal{M}$. As mentioned, for the construction of the approximating operators we assume we have represented all operators on the Hilbert space $L^{2}(\mathcal{M}, \tau)$; a discussion on how this is a natural assumption for our suggested method, along with the the practicality of the approximation, will be continued in Section \ref{sec:classical_systems}. 

\begin{lemma}\label{lem:conditional_expectations}
Let $\mathcal{M}$ be an injective von Neumann algebra with a faithful, normal tracial state $\tau$. There exists a net $(E_{F})_{F \in \Lambda}$ of trace-preserving conditional expectations on $\mathcal{M}$ for which $E_{F}(\mathcal{M})$ is finite-dimensional, and the net $(E_{F})_{F \in \Lambda}$ converges in the strong operator topology to $I_{L^{2}(\mathcal{M}, \tau)}$.
\end{lemma}
\begin{proof}
This mainly follows from the existence of a faithful, normal and trace-preserving conditional expectation onto any subalgebra of $\mathcal{M}$ (see \cite[Lemma 1.5.11]{bo}). If $\mathcal{N}\subseteq \mathcal{M}$ is a subalgebra, as $\mathcal{M}$ is injective the existence of a faithful normal conditional expectation $E_{\mathcal{N}}: \mathcal{M}\rightarrow \mathcal{N}$ ensures that $\mathcal{N}$ is injective as well; in particular, we may consider $E_{\mathcal{N}}: \mathcal{M}\rightarrow \mathcal{N}$ as the restriction to $\mathcal{M}$ of the orthogonal projection $e_{\mathcal{N}}$  of $L^{2}(\mathcal{M}, \tau)$ onto $L^{2}(\mathcal{N}, \tau)$, when $\mathcal{M}, \mathcal{N}$ are identified as subspaces in $L^{2}(\mathcal{M}, \tau)$ and $L^{2}(\mathcal{N}, \tau)$, respectively. If $F \subseteq \mathcal{M}$ is a finite subset of elements, then $W^{*}(F)$ is an injective von Neumann algebra; furthermore, as it has a separable predual results from \cite{connes} show that $W^{*}(F)$ is AFD. Thus, for any finite subset $F \subseteq \mathcal{M}$ there exists a finite-dimensional $*$-subalgebra $\mathcal{M}_{F} \subseteq \mathcal{M}$ with $d(a, \mathcal{M}_{F}) < 1/|F|$ for all $a\in F$ (where we consider the distance under the norm induced by $\tau$). Applying the existence theorem once more, we obtain a faithful, normal trace-preserving conditional expectation $E_{F}: \mathcal{M}\rightarrow \mathcal{M}_{F}$; then $\|E_{F}(a)-a\|_{\tau} = d(a, \mathcal{M}_{F})$. As the latter converges to $0$ as $|F| \rightarrow \infty$, this implies $(E_{F})_{F \in \Lambda}$ is a net of conditional expectations which form a strong approximate identity on $L^{2}(\mathcal{M}, \tau)$.  
\end{proof}

For $z > 0$, we begin by obtaining a bounded, skew-adjoint transformation of the drift term $V_{0}$ using the bounded, continuous, antisymmetric function $q_{z}: i\mathbb{R}\rightarrow \mathbb{C}$,
\begin{gather*}
    q_{z}(i\omega) = \frac{i\omega}{z^{2}+\omega^{2}}, \;\;\;\; {\rm rng}(q_{z}) = i\bigg[-\frac{1}{2z}, \frac{1}{2z}\bigg],
\end{gather*}
\noindent via the Borel functional calculus:
\begin{gather*}
    Q_{z} := q_{z}(V_{0}) \equiv R_{z}^{*}V_{0}R_{z}, \;\;\;\; Q_{z} \in \mathcal{B}(H), \;\;\;\; Q_{z}^{*} = -Q_{z},
\end{gather*}
\noindent where $R_{z} = R(z, V_{0})$. While $q_{z}$ is not invertible, its restriction $\tilde{q}_{z}$ on the subset $i\Omega_{z} := i(-\infty, -z]\cup i[z, \infty)$ of $i\mathbb{R}$ has inverse $\tilde{q}_{z}^{-1}: \bigg[-\frac{1}{2z}, \frac{1}{2z}\bigg]\setminus \{0\} \rightarrow i\mathbb{R}$ given by
\begin{gather*}
    \tilde{q}_{z}^{-1}(i\omega) = i\frac{1+\sqrt{1-4z^{2}\omega^{2}}}{2\omega}.
\end{gather*}
\noindent By injectivity of $\mathcal{M}$, take a net $(E_{F})_{F \in \Lambda}$ of conditional expectations on $\mathcal{M}$ satisfying the conditions of Lemma \ref{lem:conditional_expectations}. As $e_{F}$ is a contraction for each $F \in \Lambda$, $\sigma(E_{F}Q_{z}E_{F}) \subseteq [-1/z, 1/z]$. We then define
\begin{gather*}
    (V_{0})_{z, F} := \tilde{q}_{z}^{-1}(E_{F}Q_{z}E_{F}).
\end{gather*}
\begin{theorem}
For $z > 0, F \in \Lambda$ the operators $(V_{0})_{z, F}$ are skew-adjoint elements in $L^{0}(\mathcal{M}, \tau)$, with $(V_{0})_{z, F} \rightarrow_{\rm SRT} V_{0}$ under the iterated limits $z \rightarrow 0^{+}$ after $F \rightarrow \infty$. 
\end{theorem}\label{thm:approximating_drift}
\begin{proof}
As $V_{0}$ is skew-adjoint and affiliated to $\mathcal{M}$, and as $q_{z}$ is a bounded, continuous (and hence Borel measurable) over $\sigma(V_{0}) \subseteq i\mathbb{R}$, by Lemma \ref{lemma:closure_under_borel} we have $Q_{z} = q_{z}(V_{0}) \in \mathcal{M}$, as $Q_{z}$ is a bounded operator. Furthermore, as $e_{F}$ is a projection in $\mathcal{M}$ for each $F \in \Lambda$, $e_{F}Q_{z}e_{F} \in \mathcal{M}_{F} \subseteq \mathcal{M}$. Each $(V_{0})_{z, F}$ is skew-adjoint, as $q_{z}$ is anti-symmetric, $e_{F}$ is self-adjoint, and $V_{0}$ is skew-adjoint. Note that as $\tilde{q}_{z}^{-1}$ is continuous it is Borel measurable, and hence another application of Lemma \ref{lemma:closure_under_borel} yields that $(V_{0})_{z, F} \; \eta \; \mathcal{M}$. 

Finally, we show that $(V_{0})_{z, F}\rightarrow_{\rm SRT} V_{0}$ under the iterated limits $z\rightarrow 0^{+}$ after $F\rightarrow \infty$. For fixed $z > 0$, we claim that $(V_{0})_{z, F}$ converges strongly to $(V_{0})_{z}$; indeed, as $(E_{F})_{F \in \Lambda}$ forms a strong approximate identity with $E_{F}$ an orthogonal projection for each $F \in \Lambda$, we have that $E_{F}Q_{z}E_{F}\rightarrow Q_{z}$ in the strong operator topology on $L^{2}(\mathcal{M}, \tau)$. Additionally, the net $(E_{F}Q_{z}E_{F})_{F \in \Lambda}$ is uniformly bounded, as $\|E_{F}\| \leq 1$ and thus $\|E_{F}Q_{z}E_{F}\| \leq \|Q_{z}\|$ for all $F \in \Lambda$. This, in conjunction with the continuity of $\tilde{q}_{z}^{-1}$, implies $(V_{0})_{z, F}\rightarrow (V_{0})_{z}$ in the strong resolvent topology as $F\rightarrow \infty$. Now, apply \cite[Theorem 5]{gv_two} to conclude that $(V_{0})_{z}\rightarrow_{\rm SRT} V_{0}$ and hence, the iterated convergence holds. 
\end{proof}


\section{Classical control systems}\label{sec:classical_systems}

In this section, we indicate how the approach in previous sections for quantum control can be applied to the control of classical dynamical systems; in particular, how the affiliation relation fits naturally with the use of operator-theoretic techniques in studying classical dynamical systems via Koopman operators. 

Let $\Phi: X\rightarrow X$ be an invertible, measure-preserving map of a standard probability space $(X, \Sigma, \mu)$. The mapping $\Phi$ induces a linear action on the space of observables $L^{p}(X, \mu)$, where the dynamics act via the Koopman operator
\begin{gather*}
    U: L^{p}(X, \mu)\rightarrow L^{p}(X, \mu), \;\;\;\; Uf := f\circ \Phi,
\end{gather*}
\noindent which is an isometric isomorphism of $L^{p}(X, \mu)$ for $1 \leq p \leq \infty$. When $1 < p \leq \infty$, the Koopman operator is the dual of the Perron-Frobenius operator
\begin{gather*}
    P: L^{q}(X, \mu)\rightarrow L^{q}(X, \mu), \;\;\;\; Pf = f\circ \Phi^{-1},
\end{gather*}
\noindent for conjugate $p, q$. Under continuous time evolution, we have a measure-preserving flow $\Phi^{t}: X\rightarrow X$ with 
\begin{gather*}
x(t) = \Phi^{t}(x_{0}), \;\;\;\; x(0) = x_{0} \in X, \; t \in \mathbb{R}.
\end{gather*}
\noindent For each time $t \in \mathbb{R}$, this induces a unitary operator $U^{t}: L^{2}(X, \mu)\rightarrow L^{2}(X, \mu)$, acting via $U^{t}f(x) = f(\Phi^{t}(x))$ for $f \in L^{2}(X, \mu)$. Furthermore, the mapping $t \mapsto U^{t}$ is strongly continuous, and thus the \textit{Koopman group} $\{U^{t}\}_{t \in \mathbb{R}}$ is a one-parameter strongly continuous unitary group on $L^{2}(X, \mu)$. By Stone's Theorem \cite[Theorem 5.6.36]{kr_one} there exists a skew-adjoint (typically unbounded) operator $V: D(V) \rightarrow L^{2}(X, \mu)$ via the norm limit $Vf = \lim_{t\rightarrow 0}(U^{t}f-f)/t$ such that $U_{t} = e^{tV}$ for each $t \in \bb{R}$ (using the functional calculus). 

With generator $V$ and corresponding Koopman group $\{U^{t}\}_{t \in \mathbb{R}}$, we can frame the control of observables under the dynamics of $V$ with a bilinear system of the form
\begin{gather*}
    \frac{d}{dt}f(t) = (V+\hat{V}(t))f(t), \;\;\;\; f(0) = f_{0},
\end{gather*}
\noindent for some observable $f_{0} \in L^{2}(X, \mu)$ and where the previous derivative is taken with respect to the strong topology on $L^{2}(X, \mu)$. In this setting, $\hat{V}(t)$ would be an admissible control (in the sense of Sections \ref{sec:adm_controls} and \ref{sec:qcs_geom_control}) acting on $L^{2}(X, \mu)$. There is also a natural finite von Neumann algebra on $L^{2}(X, \mu)$ which $V$ is affiliated to (though not canonical, in that there may be other choices of von Neumann algebra), which can guide the choice of admissible controls. Let 
\begin{gather*}
    {\rm VN}(\Phi) := \{U^{t}\}_{t \in \bb{R}}'' \subseteq \cl{B}(L^{2}(X, \mu))
\end{gather*}
\noindent denote the von Neumann algebra generated by the one-parameter unitary group $\{U^{t}\}_{t \in \bb{R}}$ on $L^{2}(X, \mu)$ corresponding to flow $\Phi$. Note that as $\bb{R}$ is abelian, ${\rm VN}(\Phi)$ is an abelian von Neumann algebra. We will call this the \textit{Koopman von Neumann algebra}.

It is known (see \cite[Lemma 5.6.7]{kr_one}) that for any skew/self-adjoint operator $V$ acting on a Hilbert space $H$, there exists an abelian von Neumann algebra $\mathcal{M}\subseteq \mathcal{B}(H)$ for which $V \in L^{0}(\mathcal{M})$. This is precisely the von Neumann algebra generated by the spectral projections $e(\lambda)$ of $V$ on $H$; let $\mathcal{M}_{V}$ denote this von Neumann algebra. We show that for $V$, when considered as the generator for a one-parameter unitary group $\{U^{t}\}_{t \in \mathbb{R}}$, the von Neumann algebra $\mathcal{M}_{V}$ coincides precisely with the von Neumann algebra generated by the unitary group. In all that follows, we work with self-adjoint operators $V$ on $H$; analogous statements and proofs follow for skew-adjoint operators simply through multiplication by an imaginary factor.

\begin{lemma}\label{l_gen_unit_c_star}
Let $V$ be a generator for a one-parameter unitary group $\{U^{t}\}_{t \in \bb{R}}$ on some Hilbert space $H$. If $V$ is bounded, then the ${\rm C}^{*}$-algebras generated by $\{V, I_{H}\}$ and $\{U^{t}\}_{t \in \bb{R}}$ (considered inside $\cl{B}(H)$) coincide. 
\end{lemma}
\begin{proof}
Let $\cl{A}_{V}$ denote the ${\rm C}^{*}$-algebra generated by $\{V, I_{H}\}$ and $\cl{A}$ denote the ${\rm C}^{*}$-algebra generated by $\{U^{t}\}_{t \in \bb{R}}$. As $U^{t} = e^{tV}$ for each $t \in \bb{R}$, the functional calculus tells us $\cl{A} \subseteq \cl{A}_{V}$. 

For the converse, note that as $V$ is bounded and self-adjoint, $\cl{A}_{V}$ is abelian with $\cl{A}_{V} \cong C(\sigma(V))$ via Gelfand's Theorem (see \cite[Chapter I, Section 3]{takesaki}). Furthermore, $V \mapsto \iota$, where $\iota \in \mathcal{B}(\sigma(V))$ is the identity mapping $\iota(\omega) = \omega$ for $\omega \in \sigma(V)$ with $U^{t} \mapsto e^{t\iota}$ on $\sigma(V)$ (thus, under this correspondence $e^{t\iota}(\lambda) = e^{it\lambda}$). Let $\cl{A}_{0} \subseteq C(\sigma(V))$ be the $*$-subalgebra of finite linear combinations of $e^{t\iota}$'s; by setting $t = 0$, it is clear that $\cl{A}_{0}$ also contains the constant functions. One easily argues that $\cl{A}_{0}$ also separates points in $\sigma(V)$, and thus by Stone-Weierstrass $\overline{\cl{A}_{0}}^{\|\cdot\|_{\rm sup}} = C(\sigma(V))$. Thus, any element in $\cl{A}_{V}$ is a norm-limit of elements from $\cl{A}_{0}$; as $\cl{A}_{0} \subseteq \cl{A}$, this shows $\cl{A}_{V} \subseteq \cl{A}$. 
\end{proof}
\begin{proposition}\label{prop:equiv_of_vna}
Let $V$ be a (potentially unbounded) self-adjoint operator on a Hilbert space $H$. Then $\mathcal{M}_{V}$, the von Neumann algebra generated by the spectral projections of $V$, coincides with $\{U^{t}\}_{t \in \mathbb{R}}''\subseteq \mathcal{B}(H)$, where $V$ is the generator of the Koopman group $\{U^{t}\}_{t \in \mathbb{R}}$.
\end{proposition}
\begin{proof}
It is easy to show that $V_{\pm} := V\pm iI$ acting on $H$ have range $H$, and kernel
\begin{gather*}
    \ker(V_{+}) = \ker(V_{-}) = \{0\},
\end{gather*}
\noindent with inverses $V_{+}^{-1}, V_{-}^{-1}$ which are everywhere defined and bounded such that $\|V_{+}^{-1}\|$ and $\|V_{-}^{-1}\| \leq 1$ (see, for instance, the arguments of \cite[Lemma 16.1.1]{mvn}). Note that if $x, y \in {\rm dom}(V)$, 
\begin{eqnarray*}
    \langle V_{+}^{-1}V_{+}x, V_{-}y\rangle 
    & = &
    \langle x, V_{-}y\rangle \\
    & = &
    \langle V_{+}x, y\rangle \\
    & = &
    \langle V_{+}x, V_{-}^{-1}V_{-}y\rangle,
\end{eqnarray*}
\noindent as $V$ is self-adjoint on its domain. As the ranges of $V_{+}, V_{-}$ are all of $H$, for any $\hat{x} \in H$ there exists $x \in H$ such that $\hat{x} = V_{+}V_{-}x$. Then $V_{+}V_{-}x = (V^{2}+I)x = V_{-}V_{+}x$, implying that $[V_{+}, V_{-}] = 0$ on ${\rm dom}(V^{2})$. As $V$ is self-adjoint, ${\rm dom}(V^{2})$ is a core for $V$ in $H$. Furthermore, we have 
\begin{gather*}
    V_{+}^{-1}V_{-}^{-1}\hat{x} = V_{+}^{-1}V_{-}^{-1}V_{+}V_{-}x = V_{+}^{-1}V_{+}x = x \\= V_{-}^{-1}V_{-}x = V_{-}^{-1}V_{+}^{-1}V_{-}V_{+}x = V_{-}^{-1}V_{+}^{-1}\hat{x},
\end{gather*}
\noindent implying $[V_{+}^{-1}, V_{-}^{-1}] = 0$ now on all of $H$ (as these operators are everywhere defined). As $V_{-}^{-1} = (V_{+}^{-1})^{*}$, this shows $(V_{+}^{-1})(V_{+}^{-1})^{*} = (V_{+}^{-1})^{*}(V_{+}^{-1})$ (i.e., $V_{+}^{-1}$ is a normal operator). 
As $V$ is a (potentially unbounded) self-adjoint operator on $H$, the Spectral Theorem (see \cite[Chapter 5]{kr_one}) tells us we have a spectral resolution $\{e(\lambda)\}_{\lambda \in \bb{R}}$ corresponding to $V$, in that $V = \int\limits_{\bb{R}}\lambda de(\lambda)$. For $n \in \bb{N}$, let $F_{n} = e(n)-e(-n)$, and set $H_{n} := F_{n}(H)$. Using the spectral projections of $V$, one may argue that $\mathcal{M}_{V} = W^{*}(\{F_{n}, VF_{n}: \; n \in \bb{N}\})$; furthermore, standard arguments show that $\mathcal{M}_{V}$ coincides with the von Neumann algebra generated by $I, V_{+}^{-1}$ and $V_{-}^{-1}$ inside $\mathcal{B}(H)$ (see, for instance, \cite[Theorem 5.6.18]{kr_one}). As we have shown that $V_{+}^{-1}$ is normal, this implies $\mathcal{M}_{V}$ is abelian.

We claim $V \in L^{0}(\mathcal{M}_{V})$. For if $U \in U(\mathcal{M}_{V}')$ is unitary, for any $f \in {\rm dom}(V)$ we have 
\begin{gather*}
    Uf = UV_{+}^{-1}(V+iI)f = V_{+}^{-1}U(V+iI)f,
\end{gather*}
\noindent implying $(V+iI)Uf = U(V+iI)f$. From this, we have $U^{*}(V+iI)U = V+iI$, and so $U^{*}VU = V$ on ${\rm dom}(V)$. As our choice of $U \in U(\mathcal{M}_{V}')$ was arbitrary, this shows $V \in L^{0}(\mathcal{M}_{V})$. 

Our final claim is that $\mathcal{M}_{V}= \{U^{t}\}_{t \in \mathbb{R}}''$; let $\mathcal{M}_{U}$ denote the latter von Neumann algebra. Previous comments tell us the von Neumann algebra generated by $VF_{n}|_{H_{n}}$ is just $\mathcal{M}_{V}F_{n}$ acting on $H_{n}$, for each $n \in \bb{N}$. By \cite[Corollary 5.6.31]{kr_one} we have $U^{t}F_{n}|_{H_{n}} = e^{-itVF_{n}}$ for $n \in \bb{N}$. As $VF_{n}$ is bounded for each $n \in \bb{N}$, Lemma \ref{l_gen_unit_c_star} tells us that the ${\rm C}^{*}$-algebras $\cl{M}_{V}F_{n}$ acting on $H_{n}$ and that generated by $\{U^{t}F_{n}: t \in \bb{R}\}$ on $H_{n}$ coincide; use $\cl{A}_{n}$ to denote the latter. 

If $E \in \cl{P}(\mathcal{M}_{V})$ is an arbitrary projection such that $E \leq F_{n}$ for some $n \in \bb{N}$, then $EF_{n} = F_{n}E = E$, with $E \in \mathcal{M}_{V}F_{n} = \cl{A}_{n}$; additionally, $E \in \cl{A}_{n'}$ for $n' \geq n$. Let $\{P_{\alpha}^{n}\}_{\alpha} \subseteq \mathcal{Z}(\mathcal{M}_{U}')$ be the collection of all (central) projections in $\mathcal{M}_{U}'$ such that $P_{\alpha}^{n}F_{n} = 0$ for all $\alpha$, and set $P_{n} := \vee_{\alpha}P_{\alpha}$. Then $P_{n}^{\perp} = (I-P_{n}) \in \mathcal{M}_{U}$. By \cite[Proposition 5.5.5]{kr_one}, there exists a $*$-isomorphism $\cl{M}_{U}F_{n}\rightarrow \cl{M}_{U}(I-P_{n})$, and thus an isometry (see \cite[Corollary I.5.4]{takesaki}). With the same choice of $E$ as before, assume $E = S_{n}F_{n}$ for some $S_{n} \in \cl{M}_{U}$. As a projection, we have 
\begin{gather*}
    1 = \|E\| = \|S_{n}F_{n}\| = \|S_{n}(I-P_{n})\|, 
\end{gather*}
\noindent (using the isometric property of the $*$-isomorphism), with $S_{n}(I-P_{n}) \in \cl{M}_{U}$ by our choice of $S_{n}$. Furthermore, $S_{n}(I-P_{n})F_{n} = S_{n}F_{n} = E$; as $\|S_{n}(I-P_{n})\| = 1$, without loss of generality we may assume that $\|S_{n}\| = 1$ itself. By construction of each $F_{n}$, we have that $\cup_{n=1}^{\infty}H_{n}$ is dense in $H$, and $S_{n} \in (\cl{M}_{U})_{1}$ (the unit ball) for each $n \in \bb{N}$; thus, $E \in \overline{(\cl{M}_{U})_{1}}^{\rm SOT}$. This means $E \in \cl{M}_{U}$, and therefore $F_{n} \in \cl{M}_{U}$ for each $n \in \bb{N}$. Similarly, the spectral resolution of $VF_{n}$ (and hence, $VF_{n}$ itself) is in $\cl{M}_{U}$ for each $n \in \bb{N}$. As $\cl{M}_{V}$ is generated by $F_{n}, VF_{n}$ for $n \in \bb{N}$, this shows $\cl{M}_{V} \subseteq \cl{M}_{U}$. 

For the inclusion $\cl{M}_{U} \subseteq \cl{M}_{V}$, as $f(V)$ is a well-defined element in $\cl{M}_{V}$ for any (bounded) Borel measurable function $f$ defined on $\sigma(V)$ by the functional calculus it is immediate that $U^{t} \in \cl{M}_{V}$ for each $t \in \bb{R}$. Thus, $\cl{M}_{U} \subseteq M_{V}$, and therefore $\cl{M}_{U} = \cl{M}_{V}$.   
\end{proof} 
\begin{remark}
\rm We clarify what the approximations constructed in Theorem \ref{thm:approximating_drift} look like, when applied to $V_{0} = V$ acting on $L^{2}(X, \mu)$ and affiliated to ${\rm VN}(\Phi)$. By the general Spectral Theorem, we know there exists a unitary $\tilde{U}: L^{2}(X, \mu)\rightarrow L^{2}(\sigma(V), \nu)$ (where $\nu$ is the spectral measure on $\sigma(V)$) such that $\hat{U}V\hat{U}^{*} = M_{\iota}$, where $\iota$ is as above (appropriately modified for skew-adjoint $V$, where $\sigma(V) \subseteq i\mathbb{R}$). In this case, ${\rm VN}(\Phi) \cong L^{\infty}(\sigma(V), \Sigma, \nu)$ spatially, under the unitary $\hat{U}$. Using integration against $\nu$ as a (semi)-finite trace $\tau_{\nu}$ on $L^{\infty}(\sigma(V), \Sigma, \nu)$, the GNS construction yields that $L^{2}(L^{\infty}(\sigma(V), \Sigma, \nu), \tau_{\nu}) = L^{2}(\sigma(V), \Sigma, \nu)$, and thus we may assume $L^{\infty}(\sigma(V), \Sigma, \nu)$ is in the standard representation. Furthermore, we note that affiliation properties are preserved under spatial isomorphism of von Neumann algebras: specifically, that
\begin{gather*}
    L^{0}({\rm VN}(\Phi)) \cong L^{0}(L^{\infty}(\sigma(V), \Sigma, \nu)),
\end{gather*}
\noindent again under conjugation by $\hat{U}$.

As $L^{\infty}(\sigma(V), \Sigma, \nu)$ is abelian, it is injective. Injectivity of $L^{\infty}(\sigma(V), \Sigma, \nu)$ implies the existence of a classical filtration $\Sigma = \cup_{n=1}^{\infty}\Sigma_{n}$ preserving the measure $\nu$, for which the corresponding von Neumann algebra $L^{\infty}(\sigma(V), \Sigma_{n}, \nu)$ with respect to $\Sigma_{n}$ is finite-dimensional. We realize each $L^{\infty}(\sigma(V), \Sigma_{n}, \nu) \hookrightarrow L^{\infty}(\sigma(V), \Sigma, \nu)$ as a subalgebra. Then for each $n \in \mathbb{N}$, the inclusion extends to a mapping $\iota: L^{1}(\sigma(V), \Sigma_{n}, \nu)\rightarrow L^{1}(\sigma(V), \Sigma, \nu)$. By duality, this yields the classical conditional expectation $\mathbb{E}(\cdot|\Sigma_{n}): L^{\infty}(\sigma(V), \Sigma, \nu)\rightarrow L^{\infty}(\sigma(V), \Sigma_{n}, \nu)$ which acts as a projection on $L^{2}(\sigma(V), \Sigma, \nu)$. These correspond precisely to the net of conditional expectations described in Section \ref{subsec:approximation_schemes}.  

For $z > 0$ and $n \in \mathbb{N}$, let $V_{z, n}$ denote the two-parameter approximation of $V$ as constructed in Section \ref{subsec:approximation_schemes}.We note that the usage of conditional expectations $\mathbb{E}(\cdot|\Sigma_{n})$ in the construction of $V_{z, n}$ most likely yields a coarser and more computationally expensive approximation for $V$ as compared to the approximations in \cite{gv_one, gv_two}, which rely on integral operators $G_{\tau}: H\rightarrow H$ induced by particular families of kernels indexed over $\tau > 0$ on $X\times X$. It would be interesting to investigate whether a suitable choice of integral operators in $\mathcal{M}$ (in the general case) would yield better approximations than what is proposed here.
\end{remark}

\subsection{Control systems on manifolds and RKHAs}\label{ss:control_rkha} We indicate in this section how our approach to control problems using affiliated operators, in conjunction with the properties of reproducing kernel Hilbert algebras, can be used to transfer results on controllability in the Hilbert space setting to results on controllability in the underlying manifold. In what follows, we assume further that $X$ is a $C^{\infty}$-manifold, that $\mathcal{H} \subseteq C^{\infty}(X)$ is an RKHA over $X$ with kernel $k:X\times X\rightarrow \mathbb{C}$, that $\mathcal{H}$ separates points of $X$, and that $\{\nabla f(x): \; f \in \mathcal{H}\}$ spans the cotangent space $T_{x}^{*}(X)$ for every point $x \in X$. As the context is clear, we use $\langle \cdot, \cdot\rangle$ to denote the inner product $\langle \cdot, \cdot\rangle_{\mathcal{H}}$ in $\mathcal{H}$. We remark that such assumptions on an RKHA are natural to consider, as standard examples of RKHA's over compact smooth manifolds satisfy such properties (see \cite{gm_rkha, grch}). 

For $x \in X$, we consider a system of the form
\begin{gather}\label{eqn:rkha}
    \frac{d}{dt}\xi(t) = (V+V(t))\xi(t), \;\;\;\; \xi(0) = k_{x}
\end{gather}
\noindent which evolves in $\mathcal{H}$. We place the following restrictions on the operators arising in (\ref{eqn:rkha}):
\begin{itemize}
    \item[(A1)] $V, V(t)$ are closed, skew-adjoint operators on $\mathcal{H}$ with $V, V(t) \in L^{0}(\mathcal{M})$ for some countably decomposable finite type $\mathcal{M}$ on $\mathcal{H}$, and $V = \overrightarrow{V}\cdot \nabla$, $V(t) = \overrightarrow{V}(t)\cdot\nabla$ where $\overrightarrow{V}, \overrightarrow{V}(t)$ are smooth vector fields on $X$ for all choice of $t$;
    \item[(A2)] The mapping $t \mapsto (V+V(t))$ is piecewise strongly continuous on $\mathcal{H}$ for a finite number of intervals $[t_{i}, t_{i+1}]$ with 
    $i = 1, \hdots, n$ and $t \in [t_{i}, t_{i+1}]$, and $t \mapsto (V+V(t))\xi$ is differentiable over each interval $[t_{i}, t_{i+1}]$ for each fixed $\xi \in {\rm dom}(V+V(t))$. 
\end{itemize}
\begin{remark}
\rm Assumption (A1) guarantees there exists a dense subspace in $\mathcal{H}$ on which $V+V(t)$ is defined for all time $t \in \mathbb{R}$; let $D \subseteq \mathcal{H}$ denote this (time-invariant) dense domain.  
\end{remark}
\noindent With assumptions (A1), (A2) results in \cite{kato_one} show that (up to an imaginary scaling) for each time interval $[t_{i}, t_{i+1}], i = 1, \hdots, n$ there exist a collection of unitaries $U(t, s): \mathcal{H}\rightarrow \mathcal{H}$ which for all $r, s, t \in [t_{i}, t_{i+1}]$ with $r \leq s \leq t$ satisfy the following properties:
\begin{itemize}
    \item[(i)] $U(s, t)^{*} = U(t, s)$;
    \item[(ii)] $U(t, t) = I_{\mathcal{H}}$;
    \item[(iii)] $U(t, s)U(s, r) = U(t, r)$;
    \item[(iv)] Strong derivatives $\frac{\partial}{\partial t}U(t, s)$ and $\frac{\partial}{\partial s}U(t, s)$ exist, with 
    \begin{gather*}
        \frac{\partial}{\partial t}U(t, s) = -U(t, s)(V+V(t)), \;\;\;\; \frac{\partial}{\partial s}U(t, s) = (V+V(s))U(t, s).
    \end{gather*}
\end{itemize}
\begin{remark}\label{remark:prop_of_adjoints}
\rm Using properties of the unitaries (i)-(iv) above, one may also show that strong derivatives $\frac{\partial}{\partial s}U(s, t)^{*}$ and $\frac{\partial}{\partial t}U(s, t)^{*}$ exist, with
\begin{gather*}
    \frac{\partial}{\partial t}U(s, t)^{*} = (V+V(t))U(s, t)^{*}, \;\;\;\; \frac{\partial}{\partial s}U(s, t)^{*} = -U(s, t)^{*}(V+V(s)).
\end{gather*}
\noindent Additionally, by (A1) above and tracing through the explicit construction of $U(t, s)$, we have that $U(t, s) \in \mathcal{M}$ for all $s, t \in \mathbb{R}$. 
\end{remark} 

\begin{proposition}\label{prop:mult_unitaries}
Under the previous assumptions on (\ref{eqn:rkha}), for each fixed interval $[t_{i}, t_{i+1}], i = 1, \hdots, n$ and all $s, t \in [t_{i}, t_{i+1}]$ with $s \leq t$ we have that $U^{*}(s, t)$ is an algebra homomorphism on $\mathcal{H}$.  
\end{proposition}
\begin{proof}
The proof uses ideas from the proof of \cite[Proposition 2.7]{el}. Fix $s \in \mathbb{R}$, and $f, g \in D$. As a matter of notation, set $\hat{V}(t) := V+V(t)$ for $t \in \mathbb{R}$. Define the mapping $\alpha_{s}: [s, t_{i+1}]\rightarrow \mathcal{H}$ via $\alpha_{s}(t) := (U(s, t)^{*}f)(U(s, t)^{*}g)$ for $t \geq s$ in $[t_{i}, t_{i+1}]$. Let $h \in \mathcal{H}$. We will show that the mapping $t \mapsto \langle \alpha_{s}(t), h\rangle$ is differentiable with
\begin{gather}\label{eqn:cauchy_der}
    \frac{d}{dt}\langle \alpha_{s}(t), h\rangle = \bigg\langle (\hat{V}(t)U(s, t)^{*}f)(U(s, t)^{*}g)+(U(s, t)^{*}f)(\hat{V}(t)U(s, t)^{*}g), h\bigg\rangle,
\end{gather}
\noindent for all $t \geq s$ in $(t_{i}, t_{i+1})$. Let $\epsilon > 0$ such that $s \leq t-\epsilon < t < t+\epsilon \leq t_{i+1}$, and first note that
\begin{eqnarray*}
    \frac{1}{\epsilon}(\langle \alpha_{s}(t+\epsilon), h\rangle-\langle \alpha_{s}(t), h\rangle)
    & = &
    \frac{1}{\epsilon}\bigg(\langle (U(s, t+\epsilon)^{*}f)(U(s, t+\epsilon)^{*}g), h\rangle-\langle (U(s, t)^{*}f)(U(s, t)^{*}g), h\rangle\bigg) \\
    & = &
    \frac{1}{\epsilon}\langle (U(s, t+\epsilon)^{*}f-U(s, t)^{*}f)(U(s, t+\epsilon)^{*}g), h\rangle \\
    & + &
    \frac{1}{\epsilon}\langle(U(s, t)^{*}f)(U(s, t+\epsilon)^{*}g-U(s, t)^{*}g), h\rangle.
\end{eqnarray*}
\noindent We wish to show that 
\begin{gather*}
    \lim\limits_{\epsilon\rightarrow 0}\frac{1}{\epsilon}\langle(U(s, t+\epsilon)^{*}f-U(s, t)^{*}f)(U(s, t+\epsilon)^{*}g), h\rangle = \langle (\hat{V}(t)U(s, t)^{*}f)(U(s, t)^{*}g), h\rangle.
\end{gather*}
\noindent To that end, we use the expression above to estimate bounds
\begin{eqnarray*}
    & & \bigg|\frac{1}{\epsilon}\bigg(\langle (U(s, t+\epsilon)^{*}f-U(s, t)^{*}f)(U(s, t+\epsilon)^{*}g), h\rangle-\langle(\hat{V}(t)U(s, t)^{*}f)(U(s, t)^{*}g), h\rangle\bigg)\bigg| \\
    & = &
    \bigg|\bigg\langle\bigg(\frac{1}{\epsilon}(U(s, t+\epsilon)^{*}f-U(s, t)^{*}f)-\hat{V}(t)U(s, t)^{*}f\bigg)\bigg(U(s,t+\epsilon)^{*}g\bigg), h\bigg\rangle \\
    & + &
    \bigg\langle \bigg(\hat{V}(t)U(s, t)^{*}f\bigg)\bigg(U(s, t+\epsilon)^{*}g-U(s, t)^{*}g\bigg), h\bigg\rangle\bigg|  \\
    & \leq &
    \tilde{C}\bigg\|\frac{1}{\epsilon}(U(s, t+\epsilon)^{*}f-U(s, t)^{*}f)-\hat{V}(t)U(s, t)^{*}f\|_{\mathcal{H}}\|U(s, t+\epsilon)g\|_{\mathcal{H}}\|h\|_{\mathcal{H}} \\
    &+&
    \bigg|\langle U(s, t+\epsilon)^{*}g-U(s, t)^{*}g, (\hat{V}(t)U(s, t)^{*}f)^{*}(h)\rangle\bigg|.
\end{eqnarray*}
\noindent Note that $\sup\limits_{\epsilon \neq 0}\|U(s, t+\epsilon)^{*}g\|_{\mathcal{H}} < \infty$. Furthermore, as $\frac{\partial}{\partial t}U(s, t)^{*} = \hat{V}(t)U(s, t)^{*}$ as a strong derivative in $\mathcal{H}$, we have that
\begin{gather*}
    \lim\limits_{\epsilon\rightarrow 0}\bigg\|\frac{1}{\epsilon}(U(s, t+\epsilon)^{*}f-U(s, t)^{*}f)-\hat{V}(t)U(s, t)^{*}f\bigg\|_{\mathcal{H}} = 0.
\end{gather*}
\noindent Additionally, as $s, t$ and $f$ are fixed, we have
\begin{gather*}
    \lim\limits_{\epsilon\rightarrow 0}\bigg|\langle U(s, t+\epsilon)^{*}g-U(s, t)^{*}g, (\hat{V}(t)U(s, t)^{*}f)^{*}(h)\rangle\bigg| = 0,
\end{gather*}
\noindent once we bring the limit inside the inner product. Together, these imply
\begin{gather*}
    \lim\limits_{\epsilon\rightarrow 0}\frac{1}{\epsilon}\langle(U(s, t+\epsilon)^{*}f-U(s, t)^{*}f)(U(s, t+\epsilon)^{*}g), h\rangle = \langle (\hat{V}(t)U(s, t)^{*}f)(U(s, t)^{*}g), h\rangle.
\end{gather*}   
\noindent A similar argument shows that
\begin{gather*}
    \lim\limits_{\epsilon\rightarrow 0}\frac{1}{\epsilon}\langle(U(s, t)^{*}f)(U(s, t+\epsilon)^{*}g-U(s, t)^{*}g), h\rangle = \langle (U(s, t)^{*}f)(\hat{V}(t)U(s, t)^{*}g), h\rangle.
\end{gather*}
\noindent Therefore, $t \mapsto \langle \alpha_{s}(t), h\rangle$ is (weakly) differentiable as claimed, with weak derivative
\begin{gather*}
    \alpha_{s}'(t) = (\hat{V}(t)U(s, t)^{*}f)(U(s, t)^{*}g)+(U(s, t)^{*}f)(\hat{V}(t)U(s, t)^{*}g).
\end{gather*}
\noindent However, as $D$ is an algebra and $\hat{V}(t)$ is a derivation on $D$, we have that
\begin{gather*}
    \alpha_{s}'(t) = (\hat{V}(t)U(s, t)^{*}f)(U(s, t)^{*}g)+(U(s, t)^{*}f)(\hat{V}(t)U(s, t)^{*}g) 
    \\= \hat{V}(t)\bigg((U(s, t)^{*}f)(U(s, t)^{*}g)\bigg) = \hat{V}(t)\alpha_{s}(t),
\end{gather*}
\noindent for all $t$ as above. We have that $\alpha_{s}(s) = fg$ by properties of $U(s, t)^{*}$, and so by the uniqueness of solutions for the Cauchy problem (\ref{eqn:cauchy_der}) and the results of \cite{kato_one} this implies $\alpha_{s}(t) = U(s, t)^{*}(fg)$ for all $t \in \mathbb{R}$. Therefore, for all $f, g \in D$ we have
\begin{gather}\label{eqn:multiplicative}
    U(s, t)^{*}(fg) = (U(s, t)^{*}f)(U(s, t)^{*}g).
\end{gather}
\noindent Now, using that $D$ is a dense subspace in $\mathcal{H}$, we conclude (\ref{eqn:multiplicative}) holds for all $f, g \in \mathcal{H}$ as claimed. 
\end{proof}

\begin{theorem}\label{thm:control_on_manifold}
Let $y \in X$. If there exists some solution $\xi(t)$ for (\ref{eqn:rkha}) such that $\xi(T) = k_{y}$ at some time $T > 0$, then there exists a corresponding flow
\begin{gather}\label{eqn:manifold}
    \frac{d}{dt}\gamma(t) = (\overrightarrow{V}+\overrightarrow{V}(t))\gamma(t), \;\;\;\; \gamma(0) = x,
\end{gather}
\noindent determined by $\overrightarrow{V}+\overrightarrow{V}(t)$ on $X$ for which $y$ is a solution at time $T$. 
\end{theorem}
\begin{proof}
By the assumptions on $\hat{V}(t) = V+V(t)$ acting on $\mathcal{H}$, we have that 
\begin{gather*}
    \xi(t) = U(t, t_{i})U(t_{i}, t_{i-1})\cdots U(t_{2}, t_{1})k_{x}, \;\;\;\; t \geq 0, \; t \in [t_{i}, t_{i+1}].
\end{gather*}
\noindent Proposition \ref{prop:mult_unitaries} shows that $U(s, t)^{*}$ acts multiplicatively on all of $\mathcal{H}$ for $s, t \in [t_{i}, t_{i+1}]$ with $s \leq t$ and $i = 1, \hdots, n$, and thus $U(t, s)$ acts as an RKHS morphism (recall Subsection \ref{ss:rkha}); thus, there exists some function (after repeated function composition) $\gamma: [0, \infty) \rightarrow X$ such that
\begin{gather}\label{eqn:xi_expression}
    \xi(t) = U(t, t_{i})U(t_{i}, t_{i-1})\cdots U(t_{2}, t_{1})k_{x} = k_{\gamma(t)}, \;\;\;\; t \geq 0.
\end{gather}

For $f \in D$, by definition of $\hat{V}(t)$ and using that it is skew-adjoint we compute
\begin{eqnarray*}
    \frac{d}{dt}\langle f, \xi(t)\rangle 
    & = &
    \langle f, \frac{d}{dt}\xi(t)\rangle \\
    & = &
    \langle f, (V+V(t))\xi(t)\rangle \\
    & = &
    -\langle (V+V(t))f, \xi(t)\rangle 
    \\
    & = &
    -\langle (\overrightarrow{V}+\overrightarrow{V}(t))\cdot \nabla f, \xi(t)\rangle.
\end{eqnarray*}
\noindent Using the reproducing kernel property and (\ref{eqn:xi_expression}), along with an application of the chain rule, we have that 
\begin{gather*}
    \frac{d}{dt}\langle f, \xi(t)\rangle = \frac{d}{dt}\langle f, k_{\gamma(t)}\rangle = \frac{d}{dt}f(\gamma(t)) = \nabla f(\gamma(t))\cdot \frac{d}{dt}\gamma(t),
    \\ -\langle(\overrightarrow{V}+\overrightarrow{V}(t))\cdot \nabla f, \xi(t)\rangle = -\langle (\overrightarrow{V}+\overrightarrow{V}(t))\cdot \nabla f, k_{\gamma(t)}\rangle = -(\overrightarrow{V}+\overrightarrow{V}(t))\cdot \nabla f(\gamma(t)). 
\end{gather*}
\noindent Thus, 
\begin{gather*}
    \nabla f(\gamma(t))\cdot\frac{d}{dt}\gamma(t) = -(\overrightarrow{V}+\overrightarrow{V}(t))\cdot\nabla f(\gamma(t)).
\end{gather*}
\noindent As our choice of $f \in \mathcal{H}$ was arbitrary, and as $\nabla f$ spans the cotangent space at each point $x \in X$ this implies
\begin{gather*}
    \frac{d}{dt}\gamma(t) = (\overrightarrow{V}+\overrightarrow{V}(t))\gamma(t).
\end{gather*}
\noindent Note that as $\mathcal{H}$ separates points, then the feature map $z \mapsto k_{z}$ is injective. This fact, together with initial condition $\xi(0) = k_{\gamma(0)} = k_{x}$, and by the assumption that there exists some $T > 0$ such that $k_{y} = \xi(T) = k_{\gamma(T)}$, imply that $y = \gamma(T)$ is a solution of (\ref{eqn:manifold}) at time $T$ as claimed. 

\end{proof}


\section{Examples and further discussion}\label{sec:examples}
In this section, we work through a few examples which may arise in classical or quantum systems in detail and relate them to previous discussion.
\begin{example}[Control for continuous-time, measure-preserving and ergodic flows with pure-point spectrum]\label{example:linear_flow_on_torus}
\rm Much of the following initial exposition can be found in \cite[Section II]{gopss}; for greater detail on some of these points, we refer the interested reader to the aforementioned paper. Let $X$ be a compact metric space, and consider a (continuous-time) classical dynamical system $\Phi^{t}: X\rightarrow X$ which is continuous, measure-preserving, ergodic, with a pure point spectrum generated by finitely many eigenfrequencies and continuous corresponding Koopman eigenfunctions. Every such system is topologically conjugate to an ergodic rotation on a $d$-dimensional torus; thus, without loss of generality we assume $X = \mathbb{T}^{d}$. If $x = (\theta_{1}, \hdots, \theta_{d})$ for some point $x \in \mathbb{T}^{d}$ (and thus $\theta_{i} \in [0, 2\pi)$ for $i = 1, \hdots, d$), the dynamics may be described by the flow map
\begin{gather}\label{eqn:linear_flow_torus}
    \Phi^{t}(x) = (\theta_{1}+\alpha_{1}t, \hdots, \theta_{d}+\alpha_{d}t)\mod 2\pi,
\end{gather}
\noindent where $\alpha_{1}, \hdots, \alpha_{d}$ are positive, rationally independent frequency parameters. Letting $\mu$ denote the (normalized) Haar measure on $\mathbb{T}^{d}$, we let $\{U^{t}\}_{t \in \mathbb{R}}$ denote the Koopman group on $L^{2}(\mathbb{T}^{d}, \mu)$ corresponding to the linear flow, with $V: {\rm dom}(V)\rightarrow L^{2}(\mathbb{T}^{d}, \mu)$ the densely-defined, skew-adjoint generator. For explicit details on the construction of the topological mapping which allows us to represent any such dynamical system $\Phi^{t}: X\rightarrow X$ as a linear $d$-torus rotation, see \cite{efhn} or \cite[Appendix D.4]{gopss}. 

Following the constructions first taken up in \cite{dg}, set parameter $p \in (0, 1)$ with $\tau > 0$. For $j = (j_{1}, \hdots, j_{d}) \in \mathbb{Z}^{d}$, write 
\begin{gather*}
    \phi_{j}(x) = \prod\limits_{m=1}^{d}\varphi_{j_{m}}(\theta^{m}), \;\;\;\; \varphi_{\ell}(\theta) = e^{i\ell \theta},
\end{gather*}
\noindent to represent the Fourier functions on $\mathbb{T}^{d}$. Define the map $|\cdot|_{p}: \mathbb{Z}^{d}\rightarrow \mathbb{R}_{+}$ via
\begin{gather*}
    |j|_{p} := |j_{1}|^{p}+\cdots+|j_{d}|^{p},
\end{gather*}
\noindent and functions $\psi_{j} \in C(X)$ via
\begin{gather*}
    \psi_{j} := e^{-\tau |j|_{p}/2}\phi_{j}, \;\;\;\; j \in \mathbb{Z}^{d}.
\end{gather*}
\noindent Then, define kernel $\tilde{k}: X\times X\rightarrow \mathbb{R}_{+}$ via the series
\begin{gather}\label{eqn:mercer}
    \tilde{k}(x, x') = \sum\limits_{j \in \mathbb{Z}^{d}}\psi_{j}^{*}(x)\psi_{j}(x'),
\end{gather}
\noindent where the sum over $j$ converges uniformly on $X\times X$ to a smooth function. We note that $\tilde{k}$ is translation-invariant on $X = \mathbb{T}^{d}$; in particular, if $e \in X$ is the identity element, the flow (\ref{eqn:linear_flow_torus}) satisfies $\Phi^{t}(x) = x+\Phi^{t}(e)$ allowing us to deduce the dynamical invariance property
\begin{gather*}
    \tilde{k}(\Phi^{t}(x), \Phi^{t}(x')) = \tilde{k}(x, x'), \;\;\;\; x, x' \in X, \; t \in \mathbb{R}.
\end{gather*}
\noindent It was shown in \cite{dg} (see also \cite{feich, grch}) that $\tilde{k}$ is a strictly positive-definite kernel on $X$ for any $p \in (0, 1)$ and $\tau> 0$, and thus it induces an RKHS $\mathcal{H}$ on $X$, dense in $C(X)$, which separates points over $X$. Furthermore, the functions $\{\psi_{j}: \; j \in \mathbb{Z}^{d}\}$ as defined above form an orthonormal basis of $\mathcal{H}$, and thus every $f \in \mathcal{H}$ has an expansion
\begin{gather*}
    f = \sum\limits_{j \in \mathbb{Z}^{d}}\tilde{f}_{j}\psi_{j} = \sum\limits_{j \in \mathbb{Z}^{d}}\tilde{f}_{j}e^{-\tau|j|_{p}/2}\phi_{j},
\end{gather*}
\noindent where the sum over $j$ converges in the RKHS-norm. In fact, it is shown that $\mathcal{H}$ also has reproducing kernel Hilbert algebra structure, arising from the fact that the weight $\lambda \in L^{1}(\mathbb{Z}^{d})$ given by $\lambda(j) = e^{\tau |j|_{p}/2}$ for $p \in (0, 1)$ and $\tau > 0$ is subconvolutive (in that $\lambda\ast \lambda(j) \leq C\lambda(j)$ for all $j \in \mathbb{Z}^{d}$, where $C$ is some constant--- see also \cite{gm}). Results from \cite[Section IV]{gopss} show that $\{U^{t}\}_{t \in \mathbb{R}}$ and $V$ are well-defined on $\mathcal{H}$, as restrictions on the appropriate domain to $\mathcal{H} \subseteq C(X)$; we also may represent $V = \overrightarrow{V}\cdot\nabla$, where $\overrightarrow{V}$ is a smooth vector field on $\mathbb{T}^{d}$. One way to see that $\{\nabla f(x): \; f \in \mathcal{H}\}$ spans $T_{x}^{*}(X)$ for every $x \in X$ follows from the Mercer expansion (\ref{eqn:mercer}): as $\psi_{j}(x)^{*}\psi_{j}(x') = e^{-\tau|j|_{p}}\phi_{j}(x)^{*}\phi_{j}(x')$ with $e^{-\tau|j|_{p}} > 0$ for all $j$, and as $\phi_{j}$ denotes a Fourier function on compact manifold $\mathbb{T}^{d}$ we can use a standard heat-kernel argument to conclude the desired property holds. 

Let $\mathcal{M} \equiv \mathcal{M}_{V}$ denote the von Neumann algebra on $\mathcal{H}$ generated by the spectral projections of $V$; by Proposition \ref{prop:equiv_of_vna}, this is equivalent to the abelian Koopman von Neumann algebra. By construction, $\mathcal{H}$ is a separable Hilbert space, and thus $\mathcal{M}$ is finite-type (as it is abelian) and countably decomposable. Furthermore, $V \in L^{0}(\mathcal{M})$. We set up a control problem for the classical dynamical system on $\mathcal{H}$ as described in Section \ref{sec:classical_systems} by considering the equation
\begin{gather}\label{eqn:control_on_torus}
    \frac{d}{dt}\xi(t) = (V+\sum\limits_{j=1}^{d}u_{j}(t)V_{j})\xi(t), \;\;\;\; \xi(0) = \tilde{k}_{x},
\end{gather}
\noindent where $u_{j}: \mathbb{R}\rightarrow \mathbb{R}$ are all piecewise constant and bounded, $V_{j} = \frac{\partial}{\partial \theta_{j}}\cdot\nabla$ is the $j^{\rm th}$-directional derivative over $\mathbb{T}^{d}$ for $j = 1, \hdots, d$, and $\tilde{k}_{x} \in \mathcal{H}$ kernel section corresponding to some point $x \in \mathbb{T}^{d}$. One easily checks that the spectral projections for $V_{j}$ commute with $V$, and thus $V_{j} \in L^{0}(\mathcal{M})$ for $j = 1, \hdots, d$. Therefore, condition (A1) is satisfied on RKHA $\mathcal{H}$; additionally, by our use of piecewise constant and bounded functions it is easily seen that condition (A2) is also satisfied. By virtue of Theorem \ref{thm:control_on_manifold}, if $\tilde{k}_{y}$ is a solution to (\ref{eqn:control_on_torus}) for some other point $y \in \mathbb{T}^{d}$, this tells us about controllability of the original linear flow on $\mathbb{T}^{d}$. Note that, while $\mathcal{M}$ is abelian (and hence, not a factor), the RKHA structure and the orthonormal basis $\{\psi_{j}: \; j \in \mathbb{Z}^{d}\}$ of $\mathcal{H}$ can be used to show that there exists no projection $p \in \mathcal{P}(\mathcal{M})$ for which kernel section $\tilde{k}_{x} \in p\mathcal{H}$ at point $x \in \mathbb{T}^{d}$. Thus, it is reasonable to expect some manner of state controllability for the kernel sections corresponding to $\tilde{k}$ which induces $\mathcal{H}$. 
\end{example}

\begin{example}
\rm We discuss a natural class of quantum control systems for which the Hamiltonians are affiliated to a finite von Neumann algebra: these are ones which admit ``abelian symmetries" in some sense, which we make precise in a moment. A prime candidate for systems displaying this type of behavior are multi-level atoms interacting with a cavity (usually described by a harmonic oscillator via a quadratic interaction term); we follow the exposition in \cite{kzs}, but refer the interested reader to \cite{bbr, rbmb, yl}. 

Start by taking a strongly continuous unitary representation $\pi: U(1) \rightarrow U(\mathcal{B}(H))$; such a representation can be written in terms of a (potentially unbounded) self-adjoint operator $V$ acting on $H$ with pure point spectrum $\sigma(V) \subseteq \mathbb{Z}$; we write $e^{i\alpha} \mapsto e^{i\alpha V}$ for $\pi$. We further require:
\begin{itemize}
    \item[(i)] all eigenvalues of $V$ are of finite multiplicity;
    \item[(ii)] all eigenvalues of $V$ are non-negative.
\end{itemize}
\noindent The latter condition can be dropped without much effort, but the first is necessary. In this case, letting $V^{(n)}$ denote the eigen-projection of $V$ corresponding to eigenvalue $n \geq 0$, the eigenspace $H^{(n)} = V^{(n)}H$ is finite-dimensional, and we have the decomposition $H = \oplus_{n=0}^{\infty} H^{(n)}$. In \cite{kzs} they pick 
\begin{gather*}
    D_{V} := \{\psi \in H:\; V^{(n)}\psi = 0 \; {\rm for \; all \; but \; finitely \; many \;}n\},
\end{gather*}
\noindent what they call the space of ``finite particle vectors", as a dense domain of interest. This setup is used to prove the following theorem, which we include for the sake of clarity:
\begin{theorem}[Theorem 2.1, \cite{kzs}]
Consider a strongly continuous representation $\pi$ of $U(1)$ on $H$, with self-adjoint operator $V$ satisfying the previously stated conditions. Then the following statements hold:
\begin{itemize}
    \item[(i)] A self-adjoint operator $W$ commuting with $V$ admits $D_{V}$ as an invariant domain. Hence, the space $\mathfrak{u}(V) = \{iW: \; W = W^{*} {\rm \;commuting \; with 
    \;}V\}$ is a Lie algebra with the commutator as its Lie bracket;
    \item[(ii)] The exponential map is well defined on $\mathfrak{u}(V)$, and maps it onto the strongly closed subgroup $U(V) = \{U \in U(\mathcal{B}(H)): \; [U, \pi(z)] = 0, \; z \in U(1)\}$ of $U(\mathcal{B}(H))$;
    \item[(iii)] The subalgebra $\mathfrak{l} \subseteq \mathfrak{u}(V)$ generated by a family of Hamiltonians $iH_{1}, \hdots, iH_{N}$ in $\mathfrak{u}(V)$ is mapped by the exponential map into $\mathcal{G}$, and the strong closure of ${\rm exp}(\mathfrak{l})$ coincides with $\mathcal{G}$. 
\end{itemize}
\end{theorem}
We note that by placing such restrictions on $V$ (which we may think of as the drift term $V_{0}$ in our setup), and requiring the Hamiltonians to commute with $V$ (i.e., for all of their corresponding spectral projections to commute with the spectral projections of $V$), they are asking for $H_{1}, \hdots, H_{N} \; \eta \; \{V\}''$ acting on $H$. As $\{V\}'' \cong \ell^{\infty}(\mathbb{N})$ (or $\ell^{\infty}(\mathbb{Z})$, if we drop the positivity conditions on the eigenvalues of $V$), this is an abelian von Neumann algebra. The statement of Theorem 2.1 (similarly, Proposition 4.1, 4.2, 4.4, and 4.5) in \cite{kzs} almost automatically follows in our setup, with an application of \cite[Theorem 4.6, Remark 4.7]{am_one} and Remark \ref{rem:completely_dense}. 

To see how these conditions apply to the case of multi-level atoms interacting with a cavity, we consider the simplest case using only a single atom. The Hilbert space is $H = \mathbb{C}^{2}\otimes L^{2}(\mathbb{R})$ where $L^{2}(\mathbb{R})$ is taken with the Lebesgue measure (suppressing notation). We use $e_{i}\otimes \psi_{n}$ as an orthonormal basis for $H$, where $\{e_{i}\}_{i=1}^{2}$ is the canonical orthonormal basis of $\mathbb{C}^{2}$, and $\{\psi_{n}\}_{n=0}^{\infty}$ are the Hermite polynomials. Letting
\begin{gather*}
    \sigma_{1} = \begin{bmatrix} 0 & 1 \\ 1 & 0 \end{bmatrix}, \;\;\;\; \sigma_{2} = \begin{bmatrix} 0 & -i\\ i & 0\end{bmatrix}, \;\;\;\; \sigma_{3} = \begin{bmatrix} 1 & 0 \\ 0 & -1\end{bmatrix}, \;\;\;\; \sigma_{\pm} := \sigma_{1}\pm i\sigma_{2},
    \\ a: {\rm dom}(a) \rightarrow L^{2}(\mathbb{R}), \;\;\;\; a^{*}: {\rm dom}(a^{*}) \rightarrow L^{2}(\mathbb{R}),
\end{gather*}
where the former are the Pauli matrices on $\mathbb{C}^{2}$ and the latter denote the annihilation and creation operators on $L^{2}(\mathbb{R})$, the operator $V = \sigma_{3}\otimes 1+1\otimes a^{*}a$ acts as the generator for the strongly continuous representation of $U(1)$ on $\mathbb{C}^{2}\otimes L^{2}(\mathbb{R})$. the dynamics of the control system are described by the Jaynes-Cummings model (see \cite{jc}), with
\begin{gather*}
    H_{{\rm JC}} = \omega_{A}H_{{\rm JC}, 1}+\omega_{I}H_{{\rm JC}, 2}+\omega_{C}H_{{\rm JC}, 3},
\end{gather*}
\noindent where 
\begin{gather*}
    H_{{\rm JC}, 1} := (\sigma_{3}\otimes 1)/2, \; H_{{\rm JC}, 2} := (\sigma_{+}\otimes a+\sigma_{-}\otimes a^{*})/2, \; H_{{\rm JC}, 3} = 1\otimes a^{*}a.
\end{gather*}
\noindent Values $\omega_{A}, \omega_{I}, \omega_{C}$ are frequencies which are assumed to be non-vanishing but otherwise arbitrary constants. Note that our choice of basis for $\mathbb{C}^{2}\otimes L^{2}(\mathbb{R})$ from before is a common dense domain for $V, H_{{\rm JC}, i}, i = 1, 2, 3$. In this case, it is easy to see that the Hamiltonians commute with the spectral projections of $V$, and thus all are affiliated with $\ell^{\infty}(\mathbb{N})$. This additional structure is then used to prove a variety of statements on pure-state and strong controllability of the quantum system; see \cite[Section 5]{kzs} and the statements therein. 
\end{example}

\begin{example}[A non-example]
\rm Let $H = L^{2}(\bb{R})$ with the Lebesgue measure (suppressing notation), and set $\hat{x}, \hat{p}$ as the standard representation for position and momentum as densely-defined operators acting on $L^{2}(\bb{R})$. Set
\begin{gather*}
    V_{0} := -\frac{i}{2}(\hat{p}^{2}-\hat{x}^{2}), \;\;\;\; V_{+} := \frac{1}{\sqrt{2}}(\hat{p}-\hat{x}), \;\;\;\; V_{-} :=-\frac{1}{\sqrt{2}}(\hat{p}+\hat{x}),
    \\ V_{1} := V_{+}-V_{-}, \;\;\;\;V_{2} := i(V_{+}+V_{-}).
\end{gather*}
Suppose we have a system governed by the dynamics
\begin{gather*}
    \frac{d}{dt}\xi(t) = (V_{0}+u_{1}(t)V_{1}+u_{2}(t)V_{2})\xi(t), \;\;\;\; \xi(0) = \xi_{0}, 
\end{gather*}
\noindent where $\xi_{0}, \xi(t) \in L^{2}(\bb{R})$ are unit vectors. Note that $V_{1} = \frac{2}{\sqrt{2}}\hat{p}$, and $V_{2} = -\frac{2i}{\sqrt{2}}\hat{x}$. We think about this quantum mechanical system as coupling the (one-dimensional) harmonic oscillator with controlled momentum and position, where the controls are independent of one another. 

We claim that $V_{1}, V_{2}$ are not affiliated with any finite von Neumann algebra on $L^{2}(\bb{R})$. Assume towards contradiction that $V_{1}, V_{2} \; \eta \; \mathcal{M}$ for some finite $\mathcal{M}\subseteq \cl{B}(L^{2}(\bb{R}))$. By our comment above, this implies $e^{t\hat{p}}, e^{t\hat{x}} \in \mathcal{M}$ for all $t \in \bb{R}$. If $u \in U(\mathcal{M}')$, then $u$ as a unitary commutes with all $e^{t\hat{p}}, e^{t\hat{x}}$. By the latter, this would imply that $u$ must be a multiplication operator, so $u = M_{f}$ for some $f \in L^{\infty}(\mathbb{R})$. If $\lambda: \mathbb{R}\rightarrow \mathcal{B}(L^{2}(\mathbb{R}))$ denotes the left-regular representation of $\mathbb{R}$ on $L^{2}(\mathbb{R})$, one easily verifies that $\lambda(g)M_{f} = M_{\lambda(g)f}\lambda(g)$ for all $g \in \mathbb{R}$. Thus, if $M_{f}$ commutes with all $\lambda(g), g \in \mathbb{R}$, this implies $f$ must be a constant function. Hence, $\mathcal{M}'= \bb{C}I_{L^{2}(\bb{R})}$, which implies $\mathcal{M}= \cl{B}(L^{2}(\bb{R}))$; clearly, this does not possess a finite trace over all of $\mathcal{B}(L^{2}(\mathbb{R}))$, and is only semifinite. 
\end{example}
\begin{remark}
\rm Note that it is still possible to use Lie group methods in order to analyze the controllability of the system in the previous example. Densely-defined operators $V_{0}, V_{1}, V_{2}$ all share a common dense domain $D \subseteq L^{2}(\bb{R})$ of analytic vectors--- thus, $e^{tV_{j}}\xi$ has a positive radius of convergence for some $t > 0$ and each $\xi \in D$, $j = 0, 1, 2$. This means we can still manipulate the elements of the Lie group formed by the operators $e^{tV_{j}}$, for $j = 0, 1, 2$. One can compute that
\begin{gather*}
    [V_{0}, V_{+}] = V_{+}, \;\;\;\; [V_{0}, V_{-}] = -V_{-}, \;\;\;\; [V_{+}, V_{-}] = -I_{L^{2}(\bb{R})}.
\end{gather*}
\noindent Thus, $\langle V_{0}, V_{1}, V_{2}\rangle_{{\rm Lie}, \bb{R}}$ generates a four-dimensional Lie algebra. Using this Lie algebra (and its subalgebras), one can analyze the controllability of the system (similar to \cite{kurita}). 

We note here that, for a bilinear system of the form (\ref{eqn_bilinear_system}) that if the operators $V_{0}, \hdots, V_{N}$ are self-adjoint (or skew-adjoint) and affiliated to a finite von Neumann algebra $\mathcal{M}$, this is a stronger condition than only sharing a dense domain consisting of analytic vectors for all operators. Indeed: if $V_{0}, \hdots, V_{N} \; \eta \; \mathcal{M}$ for some finite $\mathcal{M}$, by \cite{mvn} and \cite[Lemma 5.1]{nelson_one} any closed, (skew)-symmetric operator affiliated with a finite von Neumann algebra automatically has a dense domain of analytic vectors; a common domain for all terms then follows from Remark \ref{rem:completely_dense}. For further analysis on approximate controllability of infinite-dimensional quantum systems over a dense domain of analytic vectors, see \cite{wtl}.
\end{remark}

\subsection*{Acknowledgments}
D.~G.\ acknowledges support from the U.S.\ Department of Defense, Basic Research Offce, under Vannevar Bush Faculty Fellowship grant N00014-21-1-2946, managed by the Office of Naval Research, and from the U.S.\ Department of Energy under grant DE-SC0025101. G.~H.\ was supported as a Postdoctoral Fellow from the first grant. Part of this work was supported by the Swedish Research Council under grant no. 2021-06594 while the authors were in residence at Insitut Mittag-Leffler in Djursholm, Sweden, during spring 2026. The authors would like to thank the institute for their hospitality.


\end{document}